\documentclass{amsart}
\usepackage{amsmath, amsthm, amssymb, amscd, graphics, eucal}
\usepackage{pinlabel}

\begin{document}

\newtheorem{theorem}{Theorem}[section]
\newtheorem{lemma}[theorem]{Lemma}
\newtheorem{proposition}[theorem]{Proposition}
\newtheorem{corollary}[theorem]{Corollary}
\newtheorem{conjecture}[theorem]{Conjecture}
\newtheorem{question}[theorem]{Question}
\newtheorem{problem}[theorem]{Problem}
\newtheorem*{vague_question}{Motivating Question}
\newtheorem*{claim}{Claim}
\newtheorem*{criterion}{Criterion}
\newtheorem*{rationality_thm}{Rationality Theorem~\ref{positive_rational}}
\newtheorem*{ab_thm}{$w=ab$ Theorem~\ref{ab_theorem}}
\newtheorem*{stability_thm}{Stability Theorem~\ref{stability_theorem}}
\newtheorem*{stairstep_thm}{Stairstep Theorem~\ref{stairstep_theorem}}
\newtheorem*{isobar_thm}{Isobar Theorem~\ref{isobar_theorem}}
\newtheorem*{slippery_conjecture}{Slippery Conjecture}

\theoremstyle{definition}
\newtheorem{definition}[theorem]{Definition}
\newtheorem{construction}[theorem]{Construction}
\newtheorem{notation}[theorem]{Notation}

\theoremstyle{remark}
\newtheorem{remark}[theorem]{Remark}
\newtheorem{example}[theorem]{Example}

\numberwithin{equation}{subsection}

\newcommand\id{\textnormal{id}}

\newcommand\homeo{\textnormal{Homeo}}
\newcommand\rot{\textnormal{rot}}
\newcommand\Z{\mathbb Z}

\newcommand\rotz{\rot^\sim}

\newcommand\W{\mathcal W}
\newcommand\R{\mathbb R}
\newcommand\RP{\mathbb{RP}}
\newcommand\C{\mathbb C}
\newcommand\Q{\mathbb Q}
\newcommand\N{\mathbb N}
\renewcommand\H{\mathbb H}

\newcommand\cl{\textnormal{cl}}
\newcommand\scl{\textnormal{scl}}
\newcommand{\length}{\textnormal{length}}
\newcommand{\area}{\textnormal{area}}

\newcommand\til{\widetilde}
\newcommand\Tri{\Delta}
\newcommand\SL{\textnormal{SL}}
\newcommand\PSL{\textnormal{PSL}}
\newcommand\1{{\bf 1}}

\newcommand{\norm}[1]{\left|#1\right|}

\title{Ziggurats and rotation numbers}
\author{Danny Calegari}
\address{DPMMS \\ University of Cambridge \\
Cambridge CB3 0WA England}
\email{dcc43@cam.ac.uk}
\author{Alden Walker}
\address{Department of Mathematics \\ Caltech \\
Pasadena CA, 91125}
\email{awalker@caltech.edu}
\date{\today}
\dedicatory{for Bill Thurston}

\begin{abstract}
We establish the existence of new 
rigidity and rationality phenomena in the theory of nonabelian group 
actions on the circle, and introduce tools to translate questions about the 
existence of actions with prescribed dynamics into finite combinatorics.
A special case of our theory gives a very short new proof of Naimi's theorem 
(i.e. the conjecture of Jankins--Neumann) which was the last step in the 
classification of taut foliations of Seifert fibered spaces. 
\end{abstract}

\maketitle

\tableofcontents

\section{Introduction}

This paper introduces new techniques and uncovers a range of new phenomena in the
topological theory of nonabelian group actions on the circle. We develop computational
tools which let us reduce subtle dynamical questions to finite combinatorics, and 
then connect these questions back to (hyperbolic) geometry and low-dimensional topology.

\subsection{A character variety for $\homeo^+(S^1)$}

Linear actions of finitely generated groups are parameterized by 
{\em character varieties}. Because a character variety is an algebraic variety, 
the characters (which capture the abstract dynamics of a representation) are
{\em polynomial functions}, and this makes answering questions about the
existence or nonexistence of actions with certain properties {\em computationally
tractable}.

It is natural to want to generalize this theory to {\em nonlinear} actions of
finitely generated groups on manifolds. In general there is probably no chance of
such a generalization; however, for the case of {\em group actions on circles},
we believe some elements of the theory can be developed.

Associated to a linear representation $\rho:\Gamma \to G$ is a trace
$\chi:\Gamma \to \R$. Given $\Gamma$ and $G$, a trace is determined by its values
on finitely many $g\in \Gamma$, and for each $g\in \Gamma$ the value $\chi(g)$
is an invariant of the conjugacy class of $\rho(g)$. Associated to a nonlinear
representation $\rho:\Gamma \to \homeo^+(S^1)$ one has the {\em rotation number}
$\rot:\Gamma \to \R/\Z$, and for each $g\in \Gamma$ the value
$\rot(g)$ is a (complete) invariant of the {\em semi-conjugacy}
class of $\rho(g)$  (i.e.\/ the equivalence relation generated by dynamical semi-conjugacy).

Given a group $G$ and a collection of elements $g_i$ and $h$ in $G$, we would like to
know what possible dynamics can be achieved by $h$ under representations of $G$ into
$\homeo^+(S^1)$ for which the $g_i$ have prescribed dynamics, either in the form of
a constraint on the rotation numbers of the $g_i$, or by imposing the condition that
the $g_i$ are conjugate to rotations through prescribed angles. A good theory should
make it possible to {\em compute} these values. In fact, it is one of the main goals
of this paper to develop tools to allow one to describe the shadows (i.e\/ the
projections under $\rot$) of such ``algebraic'' subsets
of $\text{Hom}(G,\homeo^+(S^1))$, at least for $G$ free, or a free product of cyclic groups.

\medskip

One significant point of difference between characters and rotation numbers is that
it is {\em not} true that knowing the rotation numbers of finitely many elements
suffices to determine the rest. The following example depends on knowledge of
some elementary facts about hyperbolic structures on surfaces; see e.g.\/ \cite{Katok} for
a reference.

\begin{example}
Consider a family of representations of $F_2$ into $\PSL(2,\R)$
(which acts on $\RP^1$) associated to a family of (incomplete) hyperbolic structures
on a once-punctured torus. We identify $F_2$ with the fundamental group of this torus, so that
the generators $a$ and $b$ correspond to the meridian and longitude, and the commutator $[a,b]$
corresponds to the puncture.

At the complete finite area structure, every element is
hyperbolic or parabolic, and therefore every element has rotation number $0$ (mod $\Z$).
As we deform the image of $[a,b]$ from a parabolic to an elliptic element, 
$\rot([a,b])$ becomes nonzero. However, for every finite collection of elements not conjugate
to a power of $[a,b]$, the image stays hyperbolic, and the rotation number stays zero,
for sufficiently small perturbations.
Moreover, there is a 2 parameter family of deformations of the hyperbolic structure keeping
the cone structure near the puncture fixed, and under such a deformation some elements switch
from hyperbolic to elliptic or back.
\end{example}

On the other hand, there is a certain sense in which knowledge of {\em all} rotation numbers
together with certain homological data {\em does} determine a representation, again up to the relation 
of semi-conjugacy. The precise statement is due to Ghys \cite{Ghys} and is further refined by
Matsumoto \cite{Matsumoto}, and can be expressed in terms of the ``Euler class in integral bounded
cohomology''; since a proper discussion of this would take us too far afield, we do not pursue
it here.

\medskip

If it is not true that finitely many rotation numbers determine the rest, it is nevertheless
true that the rotation numbers of (say) the generators strongly {\em constrain} the rotation numbers
of the other elements, and in quite an interesting way. 
The main problem on which we focus in this paper is therefore the following:

\begin{vague_question}
Given a free group $F$, and an element $w$ of $F$, and given values of the rotation
numbers of the generators, what is the set of possible rotation numbers of $w$?
\end{vague_question}

We are especially focussed on the special case that $w$ is a {\em positive} word in the
generators of $F$; i.e.\/ it is contained in the {\em semigroup} they generate. In this case
the theory simplifies significantly, and we are able to obtain strong results.

\subsection{Phase locking and greedy rationals}

One of the best known and best studied examples of nonlinear phase locking in dynamics is
the phenomenon of {\em Arnol{}'d tongues}: the introduction of nonlinear noise into a family of
circle homeomorphisms tends to produce {\em periodic orbits}; informally, we call this
the phenomenon of {\em greedy rationals}. In this paper we are
concerned with another manifestation of this phenomenon, in the context of {\em nonabelian
group dynamics}. We find strong constraints on the dynamics of free group actions
on the circle that maximize certain dynamical quantities, and in many case show that
these these constraints are powerful enough to guarantee periodic orbits and therefore
rational rotation numbers.

\subsection{New proofs of old theorems}

The classification of taut foliations of Seifert fibered $3$-manifolds was the
culmination of the work of many people, including Thurston, Brittenham, Eisenbud, Hirsch,
Jankins, Neumann and others, and was completed by Naimi \cite{Naimi} by proving an outstanding
conjecture of Jankins--Neumann \cite{Jankins_Neumann}. The conjecture of Jankins--Neumann
is equivalent to an analysis of the first non-trivial example in our theory, and our
methods lead to a new and very short proof of Naimi's result, thereby embedding this
classical work into a new and more powerful context.

\subsection{Statement of results}

For simplicity, the results in our paper are proved for actions of a 
free group of rank 2 on the circle. However, the generalization to arbitrary rank is
straightforward. Statements of the main theorems for arbitrary rank can be found in the
appendix; the proofs are routine generalizations of the proofs given in the paper
for the rank 2 case, and are left to the reader.

\medskip

\S~\ref{basic_definitions} is devoted to setting up notation and proving some
elementary results.

We introduce the following notation. Let $F$ be a free group of rank $2$
generated by elements $a,b$. Let $\homeo^+(S^1)^\sim$ denote the universal central
extension of the group of orientation-preserving homeomorphisms of the circle.
For $\varphi \in \homeo^+(S^1)^\sim$, let $\rotz(\varphi)$ denote the (real-valued)
rotation number.

Given $w\in F$ and real numbers $r,s$, let $X(w,r,s)$ denote the set of values 
of $\rotz(w)$ achieved by representations $F \to \homeo^+(S^1)^\sim$ for which $\rotz(a)=r$
and $\rotz(b)=s$, and let $R(w,r,s)$ denote the supremum of the set
$X(w,r,s)$. In words, $R(w,r,s)$ is the supremum of $\rotz(w)$ under all representations
for which $\rotz(a)=r$ and $\rotz(b)=s$. It turns out
(see Lemma~\ref{maximum_achieved} and Proposition~\ref{interchange}) 
that $X(w,r,s)$ is a {\em compact} interval with maximum $R(w,r,s)$ (so that the supremum is
really {\em achieved}) and minimum $-R(w,-r,-s)$, so we focus almost exclusively on $R(w,r,s)$
throughout the paper.

\medskip

A word $w\in F$ is {\em positive} if it does not contain $a^{-1}$ or $b^{-1}$.
Positive words are the focus of \S~\ref{positive_section}, and our 
strongest results are stated and proved for positive words.

The first main result we prove is the Rationality Theorem:

\begin{rationality_thm}
Suppose $w$ is positive. If $r$ and $s$ are rational, so is $R(w,r,s)$.
Moreover, if $w$ is not a power of $a$ or $b$,
the denominator of $R(w,r,s)$ is no bigger than the minimum of the denominators
of $r$ and of $s$.
\end{rationality_thm}

Moreover, we are able to reduce the computation of $R(w,r,s)$ in any case to a
{\em finite} (albeit complicated) combinatorial problem, which can be solved
explicitly, e.g.\/ by computer.

Complementing the rationality theorem is the Stability Theorem:

\begin{stability_thm}
Suppose $w$ is positive. Then $R$ is locally constant from the right at rational
points; i.e.\/ for every pair of rational numbers $r$ and $s$, there is an
$\epsilon(r,s)>0$ so that $R(w,\cdot,\cdot)$ is constant on $[r,r+\epsilon)\times[s,s+\epsilon)$.

Conversely, if $R(w,r,s)=p/q$ (where $p/q$ is reduced) and the biggest power of
consecutive $a$'s in $w$ (resp. $b$'s) is $a^m$ (resp. $b^n$), then $R(w,r+1/mq,s) \ge p/q + 1/q^2$
(resp. $R(w,r,s+1/nq) \ge p/q + 1/q^2$).
\end{stability_thm}

We think of this theorem as a nonabelian group theoretic cousin of the phenomenon
of {\em Arnol{}'d tongues}, with the ``tongues'' corresponding to the regions in the
$r$--$s$ plane where $R(w,\cdot,\cdot)$ is locally constant. The Stability Theorem shows
that the sizes of these tongues decrease with $q$, but experiments suggest a
sharper relationship $\epsilon \sim q^{-1}$.
This seems very exciting to us, and deserves further exploration.

\medskip

For the special case of $w=ab$, the combinatorics becomes simple enough to allow
a complete solution. This leads to an elementary proof of the following theorem,
first conjectured by Jankins--Neumann and proved by Naimi:

\begin{ab_thm}
For $0\le r,s < 1$ there is an equality
$$R(ab,r,s) = \sup_{p_1/q\le r, \; p_2/q \le s} (p_1 + p_2 + 1)/q$$
\end{ab_thm}

As remarked earlier, this theorem is the last step in the classification of taut foliations
of Seifert fibered spaces.

\medskip

If we fix $w$ and a rational number $p/q=\rotz(a)$, the function $R(w,p/q,\cdot)$
is nondecreasing and takes on a discrete set of values (however, it is typically discontinuous). 
The next theorem says that this function is continuous from the right, and the
points of discontinuity are rational, and can be determined by a finite procedure.

\begin{stairstep_thm}
Let $w$ be positive, and suppose we are given rational numbers $p/q$ and $c/d$ so that
$c/d$ is a value of $R(w,p/q,\cdot)$ (so necessarily $d \le q$). Then
$\inf \lbrace t : R(w,p/q,t)=c/d \rbrace$ is rational, 
and there is an algorithm to compute it. Moreover,
if $u/v$ is this infimal value, $R(w,p/q,u/v)=c/d$.
\end{stairstep_thm}

Although we are not able to give a precise description of the level sets of $R$, we
are able to give a description of the frontier of the set where $R(w,\cdot,\cdot)\ge p/q$
for a given rational number $p/q$. This is summarized in the following Isobar Theorem:

\begin{isobar_thm}
Let $w$ be positive. For any rational $p/q$ the set of $r,s\subset [0,1]\times[0,1]$ 
such that $R(w,r,s)\ge p/q$ is a finite sided rational polyhedron, whose boundary 
consists of finitely many horizontal or vertical segments.
\end{isobar_thm}

One subtle point is that the set where $R(w,r,s)=p/q$ is {\em not} in general a finite
sided rational polyhedron, and can in fact be extremely complicated.

\medskip

\S~\ref{arbitrary_word_section} moves on to the case of arbitrary $w$. The
Rationality Theorem generalizes to words $w$ which are {\em semipositive}; i.e.\/ that
contain to $a^{-1}$ or no $b^{-1}$ (Theorem~\ref{semipositive_rational}). 

We conjecture (Conjecture~\ref{rationality_conjecture}) that for arbitrary $w$
and rational $r,s$, $R(w,r,s)$ is rational.
Given $w$, a compactness argument (using circular orders
rather than actions) shows that there is a representation with $\rotz(a)=r$ and
$\rotz(b)=s$ realizing $\rotz(w)=R(w,r,s)$. Counterfactually, we suppose that $\rotz(w)$ is
irrational, and then try to modify the dynamics to increase it. This gives rise to a
dynamical problem of independent interest called the {\em interval game}. Roughly
speaking, given a finite collection of homeomorphisms $\varphi_1,\varphi_2,\cdots,\varphi_m$
and an irrational number $\theta$, we would like to find an interval $I$ and an integer $n$
for which $R_\theta^i(I^+)$ is not contained in $\varphi_j(I)$ for any $0\le i \le n$
and any $j$, but $R_\theta^n(I^+) \in I$. Here $I^+$ denotes the rightmost point of the
interval $I$, and $R_\theta$ is the rotation $R_\theta:z \to z+\theta$.

We show (Theorem~\ref{generic_smooth_enemies} and Theorem~\ref{single_enemy_nonrotation})
that the interval game can be solved for {\em generic} $C^1$ diffeomorphisms
$\varphi_i$, and for {\em all} cases of a single homeomorphism $\varphi$ except when
$\varphi$ is itself a rotation $R_\phi$. When $\varphi$ is a rotation, the set
of pairs $(\theta,\phi)$ for which the interval game can be won turns out to be
an interesting open, dense subset of the unit square, which is described in
Theorem~\ref{single_enemy_rotation}. It follows that for $w$ containing a single $a^{-1}$,
there are strong constraints on the possible irrational values of $R(w,r,s)$,
in terms of $w$ and $s$, namely that they are contained in the complement of an
explicit dense $G_\delta$.

\medskip

We are also concerned throughout this paper with representations satisfying the stronger
constraint for which $a$ and $b$ are conjugate to rotations $R_r$ and $R_s$ respectively.
We denote by $R(w,r-,s-)$ the supremum of $\rotz(w)$ over representations of this kind.
It turns out that for $w$ positive, $R(w,r-,s-)$ is the limit of $R(w,r',s')$ as $r' \to r$
and $s'\to s$ from below. Because of this, we can calculate $R(w,r-,s-)$ to any degree of
accuracy, and in many (most) cases, show that it is rational. By a trick, for rational
$r,s$ and any word $w$ there is some {\em positive} word $w'(w,r,s)$ and an explicit integer 
$N(w,r,s)$ for which $R(w,r-,s-)=N+R(w',r-,s-)$, and therefore we can (numerically)
compute $R(w,r-,s-)$ for rational $r,s$ and for {\em arbitrary} words $w$. This achieves
our goal of making the theory computationally tractable.

\section{Basic definitions and properties}\label{basic_definitions}

In the sequel we fix $F$, a free group on two generators $a$ and $b$.
We denote by $\homeo^+(S^1)$ the group of orientation-preserving homeomorphisms of
$S^1$, and by $\homeo^+(S^1)^\sim$ the group of orientation-preserving
homeomorphisms of $\R$ covering homeomorphisms of $S^1$ under the covering projection
$\R \to S^1$. There is a central extension
$$0 \to \Z \to \homeo^+(S^1)^\sim \to \homeo^+(S^1) \to 0$$
and the kernel of $\homeo^+(S^1)^\sim \to \homeo^+(S^1)$ is the group of integer
translations of $\R$.

We denote by $R_\theta\in \homeo^+(S^1)$ the homeomorphism $z \to z+\theta$; we call
$R_\theta$ the {\em rigid rotation} through angle $\theta$. By abuse of notation,
we also denote by $R_\theta$ its lift to $\homeo^+(S^1)^\sim$. In the first case,
$\theta$ should be thought of as an element of $S^1$; in the second, as an element of
$\R$. Which meaning is intended should be clear from context.

\subsection{Rotation number}

We assume the reader is familiar with Poincar\'e's rotation number; for
a basic reference, see e.g.\/ Ghys \cite{Ghys_circle} or Herman \cite{Herman}.

\medskip

Rotation number is a function $\rot:\homeo^+(S^1) \to \R/\Z$ which admits a real valued lift
$\rotz:\homeo^+(S^1)^\sim \to \R$. The following proposition summarizes some basic properties
of these functions. These properties are all well-known, and we will use them without
comment throughout the paper.

\begin{proposition}\label{elementary_properties_of_rot}
Rotation number satisfies the following properties.
\begin{enumerate}
\item{$\rot$ and $\rotz$ are conjugacy invariants, and are continuous in the $C^0$ topology.}
\item{A homeomorphism $\varphi\in\homeo^+(S^1)$ 
has rational rotation number with denominator $q$
if and only if it has a periodic orbit with period $q$.}
\item{$\rot$ and $\rotz$ are homomorphisms when restricted to any cyclic (in fact, amenable)
group of homeomorphisms.}
\item{For any $\varphi\in\homeo^+(S^1)^\sim$ with $\rotz(\varphi)=r$,
the map $t \to \rotz(R_t\circ \varphi)$
defines a continuous, nondecreasing, surjective map from $[0,1] \to [r,r+1]$
(see \cite{Herman}).}
\end{enumerate}
\end{proposition}

Some more substantial, but nevertheless classical results, concern when a homeomorphism
is conjugate to some $R_\theta$:

\begin{theorem}[Denjoy (see \cite{Herman})]
If $\varphi$ is $C^2$ and has irrational rotation number, $\varphi$ is
($C^0$) conjugate to a rigid rotation with the same rotation number.
\end{theorem}

\begin{theorem}[Herman, Yoccoz \cite{Herman,Yoccoz}]
There is a subset of the set of irrational numbers of full measure, so that if
$\theta$ is an element of this set, any $\varphi$ which is $C^\infty$ and is conjugate
to $R_\theta$, is conjugate by a $C^\infty$ diffeomorphism.
\end{theorem}

For concreteness, we call an irrational $\theta$ as in Herman--Yoccoz theorem a
{\em Herman number}.

\subsection{Extremal representations}

The fundamental question we are concerned with in this paper is the following:

\begin{question}\label{rotation_question}
Given an element $w\in F$ and real numbers $r,s$, what is the supremum 
of $\rotz(\rho(w))$ over all representations $\rho:F \to \homeo^+(S^1)^\sim$ for which
$\rotz(\rho(a))=r$ and $\rotz(\rho(b))=s$?
\end{question}

The content of Jankins--Neumann's paper \cite{Jankins_Neumann} is a partial analysis
of this question for the important but very special case $w=ab$. 
They formulated a complete conjectural answer, and went a long way towards proving it;
their conjecture was subsequently established by Naimi \cite{Naimi}.

We discuss this special case in \S~\ref{ab_case} and give a new, much shorter, and
more direct proof of Jankins--Neumann's conjecture (see Theorem~\ref{ab_theorem}). 

\begin{remark}
One can ask variations on this question for more complicated groups $G$ and elements $w\in G$;
some results in this direction are obtained in \cite{Calegari_forcing}. The
techniques and results obtained in this paper generalize in a straightforward way
to the case of a free group $F$ of rank $>2$.
\end{remark}

We introduce the following notation.
\begin{notation}
Given $w\in F$ and $r,s\in \R$, we let $R(w,r,s)$ denote the supremum of
$\rotz(\rho(w))$ over all $\rho:F \to \homeo^+(S^1)^\sim$ for which
$\rotz(\rho(a))=r$ and $\rotz(\rho(b))=s$.
\end{notation}

\begin{notation}
Let $h_a,h_b:F \to \Z$ denote the homomorphism which ``counts'' the (algebraic) number
of $a$'s or $b$'s in $w\in F$; i.e.\/ the homomorphisms defined on generators by
$h_a(a)=1$, $h_a(b)=0$ and $h_b(a)=0$, $h_b(b)=1$ respectively.
\end{notation}

\begin{lemma}\label{unique_lift}
Let $\rho:F \to \homeo^+(S^1)$ be any representation, and let $[r]=\rot(a)$, $[s]=\rot(b)$
for some $[r],[s]\in\R/\Z$.
If $r,s$ are any lifts of $[r],[s]$ to $\R$, there is a unique lift
$\rho^\Z:F \to \homeo^+(S^1)^\sim$ with $\rotz(a)=r$, $\rotz(b)=s$.
\end{lemma}
\begin{proof}
Pick any lift, then multiply the generators by suitable powers of the center.
\end{proof}

The function $R$ has some elementary properties, which we record.

\begin{lemma}[periodicity]\label{lattice_transform}
For any integers $n,m$ we have $R(w,r+n,s+m)=R(w,r,s) + nh_a(w) + mh_b(w)$.
\end{lemma}
\begin{proof}
Multiply generators by suitable powers of the center.
\end{proof}

\begin{lemma}[left-right symmetry]\label{left_right_symmetry}
If $\overline{w}$ denotes the string obtained by reversing the order of $w$, then
$R(w,r,s)=R(\overline{w},r,s)$.
\end{lemma}
\begin{proof}
Changing the orientation of the circle multiplies rotation numbers by $-1$.
Taking inverses multiplies rotation numbers by $-1$. The composition of
these two operations proves the lemma.
\end{proof}

There are {\it a priori} estimates on $R(w,r,s)$ in terms of the combinatorics of $w$.

\begin{lemma}\label{counting_inequality}
Suppose $w$ is conjugate into the form $w=a^{\alpha_1}b^{\beta_1}\cdots a^{\alpha_m}b^{\beta_m}$
for some $m=m(w)$. Then $R(w,r,s) \le m(w) + rh_a(w) + sh_b(w)$.
\end{lemma}
\begin{proof}
The proof is elementary, given some facts from the theory of stable commutator
length (denoted $\scl$, see Definition~\ref{scl_definition}); 
we will return to this connection in \S~\ref{immersion_subsection}, and
the reader can consult this section or \cite{Calegari_scl} for definitions and details.

The chain $w - a^{h_a} - b^{h_b}$ is homologically trivial, and its stable
commutator length is easily shown to satisfy the inequality
$\scl(w - a^{h_a} - b^{h_b}) \le m(w)/2$. On the other hand, the function 
$\rotz$ is a homogeneous quasimorphism with defect at most $1$ (this is
a restatement of the Milnor--Wood inequality in the language of quasimorphisms), 
so the inequality follows from generalized Bavard duality 
(\cite{Calegari_scl}, Thm.~2.79).
\end{proof}

\begin{remark}
Evidently Lemma~\ref{counting_inequality} is implied by the stronger inequality
$$R(w,r,s) \le 2\cdot\scl(w-a^{h_a}-b^{h_b}) + rh_a(w) + sh_b(w)$$ 
If $r$ and $s$ are rational, so is the right hand side of this inequality.
\end{remark}

\begin{lemma}\label{maximum_achieved}
For any $w,r,s$ the maximum $R(w,r,s)$ is achieved on some representation.
\end{lemma}
\begin{proof}
Rotation number depends (continuously) 
only on the circular order derived by the action of the group
on any orbit. The space of such circular orders is compact, and so is the subset for which
$a$ and $b$ have a specific rotation number (mod $\Z$). By Lemma~\ref{unique_lift}
and compactness, $\rotz(w)$ is maximized on some circular order (which can be realized
in many ways by an action).
\end{proof}

\begin{lemma}\label{lower_semicontinuous}
For fixed $w$, the function $R(w,r,s)$ is lower semi-continuous in
$r$ and $s$.
\end{lemma}
\begin{proof}
Lower semi-continuity means that the value of a limit is a least as big as the limit
of the values. Let $r_i,s_i \to r,s$ and let $\rho_i$ be representations
for which the maximum is achieved. Some subsequence of the $\rho_i$ converges
up to semi-conjugacy to a limit $\rho$, and $\rotz(\rho(w))=\lim_i \rotz(\rho_i(w))$.
\end{proof}

\begin{notation}
Given $w\in F$ and $r,s\in \R$, we let $X(w,r,s)$ denote the set of
values of $\rotz(\rho(w))$ over all $\rho:F \to \homeo^+(S^1)^\sim$ for which
$\rotz(\rho(a))=r$ and $\rotz(\rho(b))=s$.
\end{notation}

Note that $h_a(w)r+h_b(w)s \in X(w,r,s)$ for all $w,r,s$.
The following proposition shows that $X(w,\cdot,\cdot)$ 
can be recovered from $R(w,\cdot,\cdot)$.

\begin{proposition}\label{interchange}
For any $w$ and fixed $r,s$, the set $X(w,r,s)$ is a compact interval, with maximum $R(w,r,s)$
and minimum $-R(w,-r,-s)$.
\end{proposition}
\begin{proof}
The set of representations for which the rotation numbers of $a$ and $b$ are
prescribed is path-connected. This implies that $X(w,r,s)$ is connected, and therefore
an interval. The maximum is $R(w,r,s)$ by definition, and the minimum is 
$-R(w,-r,-s)$ because changing the orientation of the circle negates rotation number.
\end{proof}

\subsection{Approximation by smooth diffeomorphisms}

For later use, we state and prove some elementary facts about approximating
homeomorphisms by diffeomorphisms.

\begin{lemma}\label{smooth_approximation}
Any $\varphi\in \homeo^+(S^1)^\sim$ can be $C^0$ approximated by a $C^\infty$ diffeomorphism
$\varphi'\in \homeo^+(S^1)^\sim$ with the same rotation number.
\end{lemma}
\begin{proof}
First of all, it is obvious that $\varphi$ can be $C^0$ approximated by $C^\infty$
diffeomorphisms $\varphi^+$ and $\varphi^-$ in such a way that 
$|\varphi(p)-\varphi^\pm(p)|<\epsilon$ for all $p$, while $\varphi^+(p)$ is to the right
of $\varphi(p)$ and $\varphi^-(p)$ is to the left of $\varphi(p)$. One way to
do this is to approximate the graph of $\varphi$ by smooth graphs from above and below.
Evidently $\rotz(\varphi^-) \le \rotz(\varphi) \le \rotz(\varphi^+)$. 

Now for $t\in [0,1]$ define $\varphi_t(p) = t\varphi^-(p) + (1-t)\varphi^+(p)$. This
defines a smooth family interpolating between $\varphi^-$ and $\varphi^+$, so some element
of this family has the same rotation number as $\varphi$. Moreover, every element of
this family satisfies $|\varphi_t(p) -\varphi(p)|<\epsilon$ for all $p$.
\end{proof}

\begin{lemma}\label{smooth_family}
If both $r$ and $s$ are either rational or Herman numbers, every $\theta$ in the interior
of $X(w,r,s)$ is realized by a smooth representation. If both $r,s$ are Herman
numbers, every $\theta$ in the interior of $X(w,r,s)$ is realized by a representation
in which $a$ and $b$ are smoothly conjugate to $R_r$ and $R_s$ respectively.
\end{lemma}
\begin{proof}
By Lemma~\ref{smooth_approximation}, there are smooth representations with
$\rotz(a)=r$ and $\rotz(b)=s$ and for which $\rotz(w)$ is as close to any $\theta$ in $X(w,r,s)$
as we like; hence we can certainly realize a {\em dense} subset of $\theta$ in $X(w,r,s)$
by smooth representations. On the other hand, the set of diffeomorphisms whose rotation
number is equal to a fixed rational or Herman number is path connected. 

Indeed, if $\varphi$ and $\varphi'$ are both smooth with the same rational rotation number,
they both have a finite orbit with the same dynamics. We can easily find a $1$-parameter
family $\varphi_t$ interpolating between them with the same finite orbit, and hence
the same rotation number.

If $\varphi$ and $\varphi'$ are both smooth with the same irrational Herman rotation number
$r$, they are both smoothly conjugate to $R_r$. The two conjugating maps can be joined by
a smooth path of conjugating maps, exhibiting a $1$-parameter family $\varphi_t$ all
with rotation number $r$ (in fact, all conjugate to $R_r$) interpolating between them.

In conclusion, the set of $\theta \in X(w,r,s)$ realized by smooth representations is dense and
connected. The lemma follows.
\end{proof}

\begin{remark}
Shigenori Matsumoto communicated the following short proof that for
any $r$ and $s$ (not necessarily rational or Herman irrational), 
every $\theta$ in the interior of $X(w,r,s)$ is realized by a smooth representation.
This follows by showing that the space of $C^\infty$ diffeomorphisms with
a given rotation number $r$ is pathwise connected. For rational $r$ this is
obvious; for irrational $r$, let $f$ be any smooth diffeomorphism
with $\rot(f)=r$, and define 
$$f_{t,\theta}=R_\theta \circ ((1-t)f + t\,\id)$$
Because $r$ is irrational, 
for each $t$ there is a {\em unique} $\theta(t)$ so that 
$\rot(f_{t,\theta(t)})=r$; see e.g.\/ \cite{Herman}.
Since the subset $\lbrace (t,\theta(t))\rbrace$ is closed in $[0,1]\times S^1$
(being the preimage of $r$ under the continuous function $\rot$) it follows
that $t \to \theta(t)$ is continuous.
\end{remark}

\section{Positive words}\label{positive_section}

\subsection{Positivity and rationality}

\begin{definition}
A (necessarily reduced) word $w$ in $F$ is {\em positive} if it contains 
no $a^{-1}$ or $b^{-1}$ (i.e.\/ it contains only $a$'s and $b$'s).
\end{definition}

In this section we adhere to the convention that $w$ is {\em positive} unless we explicitly
say otherwise. In this case we are able to obtain complete and precise 
answers, and a surprisingly rich structure theory.

The first surprise is the following Rationality Theorem, which says that for $r,s$
rational and $w$ positive, $R(w,r,s)$ is rational, and there is an {\it a priori}
bound on its denominator. In fact, it turns out that the computation of $R(w,r,s)$
in any given case can be reduced to a finite combinatorial question!

\begin{theorem}[Rationality Theorem]\label{positive_rational}
Suppose $w$ is positive. If $r$ and $s$ are rational, so is $R(w,r,s)$.
Moreover, if $w$ is not a power of $a$ or $b$,
the denominator of $R(w,r,s)$ is no bigger than the minimum of the denominators
of $r$ and of $s$.
\end{theorem}
\begin{proof}
Consider any action of $F$ on $S^1$ for which $\rotz(a)=r$ and $\rotz(b)=s$.
We will show how to modify this action without decreasing $\rotz(w)$, until $w$
has a periodic orbit with period no bigger than that of $a$ or $b$. In fact,
the new $w$ will have a periodic orbit which can be taken to biject (in a natural way) with a
subset of a periodic orbit of either $a$ or $b$. This will prove the proposition.
We call this method of starting with any representation, and modifying the dynamics without
decreasing $\rotz(w)$ until $w$ has some desired property, the {\em method of perturbation}.

\medskip

Since $r$ and $s$ are rational, both $a$ and $b$ have finite orbits. If
$r=p_1/q_1$ then there are points $x_i$ for $0\le i \le q_1-1$ for which
$a(x_i)=x_{i+p_1}$ (indices taken mod $q_1$); similarly there are $y_j$ for 
$0\le j \le q_2-1$ for which $b(y_j) = y_{j+p_2}$ (indices mod $q_2$). 
Denote the union of the $x_i$ by $\Sigma_x$, the union of the $y_j$ by $\Sigma_y$,
and the union of both by $\Sigma$.

We can modify the dynamics of $a$ on the complement of $\Sigma_x$ (and similarly for $b$)
without changing their rotation numbers. Replacing $a$ by some new $a'$ with the
property that $a'(p) \ge a(p)$ for all $p$ cannot decrease $\rotz(w)$,
since only positive powers of $a$ appear in $w$; a similar statement holds for $b$.

Define maps $\alpha^+$ and $\beta^+$ by $\alpha^+(p) = a(x_{i+1})$ for 
$p \in (x_i,x_{i+1}]$ and $\beta^+(p)=b(y_{j+1})$ for $p\in (y_j,y_{i+1}]$.
Note that although these maps are not homeomorphisms,
they are monotone and therefore still have a well-defined rotation number.

Define $w^+$ to be the composition obtained by replacing $a$ and $b$ by $\alpha^+$ and $\beta^+$
in the word $w$. Evidently $\rotz(w^+)\ge \rotz(w)$. On the other hand, for any
$\epsilon$ we can choose $a',b'$ so that $\alpha^+(p)-a'(p)\le \epsilon$ whenever
$p-x_i>\epsilon$ (if $p \in (x_i,x_{i+1}]$) and similarly for $b'$ and $y_j$.
Set $\epsilon$ less than half the distance between distinct elements of $\Sigma$.
Then there is some initial choice of $p$ so that $|(w^+)^n(p)-w^n(p)|<\epsilon$
for any integer $n$, and therefore $\rotz(w^+) = \rotz(w)$.

\medskip

It remains to estimate the denominator of $\rotz(w^+)$. Since $w$ is by hypothesis
not a power of $a$, some conjugate of $w$ ends with $a$, and therefore the
corresponding conjugate of $w^+$
takes $\Sigma_x$ into itself. The denominator of $\rotz(w^+)$ is the least period of
an orbit, and is therefore $\le q_1$. Interchanging $a$ and $b$ gives the desired
estimate. 
\end{proof}

Note that Theorem~\ref{positive_rational} gives an {\em algorithm} to compute
$R(w,r,s)$ for positive $w$ and rational $r$ and $s$. For each configuration of $\Sigma$
in the circle, the rotation number can be read off from the map $w^+$ from
$\Sigma$ to itself. This rotation number only depends on the relative order of
the points in $\Sigma$; there are only finitely many possible configurations, so
by examining each of them in turn we can compute the maximum. This is pursued
more systematically in \S~\ref{UR_subsection}.

\begin{lemma}\label{nondecreasing}
Suppose $w$ is positive and fixed. Then $R(w,r,s)$ is monotone nondecreasing
in $r$ and in $s$.
\end{lemma}
\begin{proof}
Let $\rho$ be a representation maximizing $\rotz(w)$ for fixed $r,s$. We may increase $r$
by replacing $a$ with the composition $R_t \circ a$ where $R_t$ is a 
rotation through angle $t$. By Proposition~\ref{elementary_properties_of_rot}
bullet~(4)  we can prescribe $\rotz(R_t\circ a)\ge r$.

Since $w$ is positive, replacing $a$ with $R_t\circ a$ cannot decrease $\rotz(w)$.
\end{proof}

\begin{lemma}\label{one_rational}
Suppose $w$ is positive. If $r=p/q$ is rational, then $R(w,r,s)$ is
rational with denominator $\le q$ for all $s$.
\end{lemma}
\begin{proof}
As $s$ increases in some interval, $R(w,r,s)$ is
nondecreasing, by Lemma~\ref{nondecreasing}. On the other hand, it is rational
with denominator $\le q$ for all rational $s$ by Theorem~\ref{positive_rational}, 
and therefore by lower semicontinuity (i.e.\/ Lemma~\ref{lower_semicontinuous})
it is rational with denominator $\le q$ for all $s$.
\end{proof}

\subsection{$XY$ words and dynamics}\label{UR_subsection}

We now study the combinatorics of $\Sigma_x$ and $\Sigma_y$ (recall
the notation and setup from the proof of Theorem~\ref{positive_rational}). {\it A priori} 
it might be possible that $R(w,p_1/q_1,p_2/q_2)$ can only be achieved by some 
configuration where $\Sigma_x \cap \Sigma_y$
is nonempty. However, the following lemma shows we do not need to consider such
configurations.

\begin{lemma}\label{sigmas_disjoint}
Suppose $w$ is positive. $R(w,p_1/q_1,p_2/q_2)$ is achieved for some configuration
where $\Sigma_x$ and $\Sigma_y$ are disjoint.
\end{lemma}
\begin{proof}
Suppose we have a configuration in which $x_i=y_j$ for some $i,j$ (i.e.\/ in which
$\Sigma_x$ and $\Sigma_y$ are not disjoint). Perturb $x_i$ slightly so that
$y_{j-1} < x_i < y_j$. We apply one of $\alpha^+$ or $\beta^+$ in turn and see how our
new orbit compares to the old one. After applying some power of $\alpha^+$ we might end
up at $x_i$. But if we then apply $\beta^+$ then because $\beta^+(x_i)=\beta^+(y_j)$ 
the new orbit immediately catches up to the old.

Conversely, if we apply some power of $\beta^+$ and end up at $y_j$ and then apply
$\alpha^+$, because $\alpha^+(y_j)=\alpha^+(x_{i+1})$ the new orbit pulls ahead of 
the old. In either case, we definitely do not lag, and the rotation number is no 
smaller in the new configuration.
\end{proof}

Because of Lemma~\ref{sigmas_disjoint}, in the sequel we will assume that $\Sigma_x$ and
$\Sigma_y$ are disjoint, and therefore $|\Sigma|=q_1+q_2$. The configuration
of $\Sigma_x$ and $\Sigma_y$ in $S^1$ can be encoded (up to conjugacy) by a cyclic word $W$
in letters $X$ and $Y$ containing $q_1$ $X$'s and $q_2$ $Y$'s;
call such a word {\em admissible} for the pair $q_1,q_2$. 
The number of cyclic words admissible for $q_1,q_2$ is exponential in $\min(q_1,q_2)$, but
for fixed $q_1$ (say), is polynomial in $q_2$.


There is a dynamical system, generated by transformations $a$ and $b$, 
whose orbit space is the letters of $W$, as follows. The element $a$ acts by moving to the right
until we read off $p_1+1$ $X$'s, counting the $X$ we start on, if we start on an $X$.
The element $b$ acts by moving to the right until we read off $p_2+1$ $Y$'s, counting the $Y$
we start on, if we start on a $Y$. A maximal consecutive string $a^m$ ``hops'' over
$mp_1+1$ $X$'s, and a maximal consecutive string $b^n$ ``hops'' over $np_2+1$ $Y$'s;
we refer to these prosaically as {\em $a$-hops} and {\em $b$-hops} in the sequel.
Since $a$ and $b$ are monotone with respect to the
cyclic order, it makes sense to define the ($\Q$-valued) rotation number 
$R(w,p_1/q_1,p_2/q_2,W)$ for any admissible $q_1,q_2$ cyclic word $W$.

Lemma~\ref{sigmas_disjoint} implies
$R(w,p_1/q_1,p_2/q_2)=\max_W R(w,p_1/q_1,p_2/q_2,W)$. For convenience we also define
$r(w,p_1/q_1,p_2/q_2):=\min_W R(w,p_1/q_1,p_2/q_2,W)$ for reduced $p_1/q_1,p_2/q_2$.

\begin{example}
We compute $R(ab,2/3,1/2)$. Up to cyclic permutation, there are $2$ admissible
$3,2$ words, namely $XXXYY$ and $XXYXY$. Each application of $a$ skips over $3$ $X$'s,
and each application of $b$ skips over $2$ $Y$'s. We apply the letters of $w$ from right
to left to compute $R$; therefore we think of our group acting on $S^1$ on the left.

For the case $XXXYY$, there is an orbit of period $1$:
$$XX\bar{X}YY \xrightarrow{b} XXXY\bar{Y} \xrightarrow{a} XX\bar{X}YY$$
and the rotation number is $1$.

For the case $XXYXY$, there is an orbit of period $2$:
$$XXY\bar{X}Y \xrightarrow{b} XX\bar{Y}XY \xrightarrow{a} X\bar{X}YXY \xrightarrow{b} XXYX\bar{Y} \xrightarrow{a} XXY\bar{X}Y$$
and the rotation number is $3/2$. Hence $R(ab,2/3,1/2)=3/2$.
\end{example}

Complementing the Rationality Theorem is the following {\em Stability Theorem}:

\begin{theorem}[Stability Theorem]\label{stability_theorem}
Suppose $w$ is positive. Then $R$ is locally constant from the right at rational
points; i.e.\/ for every pair of rational numbers $r$ and $s$, there is an
$\epsilon(r,s)>0$ so that $R(w,\cdot,\cdot)$ is constant on $[r,r+\epsilon)\times[s,s+\epsilon)$.

Conversely, if $R(w,r,s)=p/q$ (where $p/q$ is reduced) and the biggest power of
consecutive $a$'s in $w$ (resp. $b$'s) is $a^m$ (resp. $b^n$), then $R(w,r+1/mq,s) \ge p/q + 1/q^2$
(resp. $R(w,r,s+1/nq) \ge p/q + 1/q^2$).
\end{theorem}
\begin{proof}
By Lemma~\ref{lower_semicontinuous} and Lemma~\ref{nondecreasing}, it follows
that $R(w,\cdot,s)$ is continuous from the right. If $s$ is rational, it takes
only finitely many values, by Theorem~\ref{positive_rational}. Hence it is
locally constant from the right.
Hence for any $r$ there is rational $r'>r$ with
$R(w,r',s) = R(w,r,s)$. Similarly, there is a rational $s'>s$ with $R(w,r',s')= R(w,r',s)$.
Monotonicity (i.e.\/ Lemma~\ref{nondecreasing}) proves the existence of an $\epsilon$
as in the statement of the theorem.

Conversely, let $W$ be admissible with
$R(w,r,s,W)=p/q$, where $r=u/v$. Each time we read $a^m$ we hop over $mu+1$ $X$'s.
In the course of a periodic orbit, 
the end of this $a$-hop lands on $q$ distinct $X$'s in $W$, which
we label $X_1,X_2,\cdots, X_q$. For some $i$ there are at most
$\lfloor v/q \rfloor$ $X$'s in the ``interval'' $(X_i,X_{i+1}]$, by the pigeonhole principle. 
Let $W'$ be obtained from $W$ by replacing each $X$ by $X^{mq}$. Then
by comparing orbits, $R(w,u/v + 1/mq,s) \ge R(w,(umq + v)/vmq,s,W') \ge p/q+1/q^2$, as claimed.
\end{proof}

It follows that for all positive $w$, there is an open, dense set in the
$r$--$s$ plane where $R(w,r,s)$ is locally constant and takes values in $\Q$. This
is a new manifestation of the familiar phenomenon of phase locking; it would be
very interesting to investigate, for a fixed $w$, 
how the maximal $\epsilon$ as in Theorem~\ref{stability_theorem} depends on $r$ and $s$.
For fixed $w$, a natural guess (in view of the inequality in the second half of
Theorem~\ref{stability_theorem}) is $\epsilon \sim q^{-1}$, where
$p/q$ is the locally constant value of $R(w,\cdot,\cdot)$. There is some experimental
evidence for this, but it seems hard to prove rigorously.

\begin{example}
The inequality in Theorem~\ref{stability_theorem} is sharp for every
$q$ for $w=a^mb^n$, as follows from Theorem~\ref{ab_theorem}, to be proved in
\S~\ref{ab_case}.
\end{example}

\subsection{The case $w=ab$}\label{ab_case}

In this section we analyze a particular important special case in detail, namely the 
case $w=ab$. We derive a concrete formula for $R(ab,p_1/q_1,p_2/q_2)$, and thereby 
give a new (and much simpler) proof of the main conjecture of 
Jankins--Neumann \cite{Jankins_Neumann}, first proved by Naimi \cite{Naimi}.

\begin{theorem}[$ab$ Theorem]\label{ab_theorem}
For $0\le r,s < 1$ there is an equality
$$R(ab,r,s) = \sup_{p_1/q\le r, \; p_2/q \le s} (p_1 + p_2 + 1)/q$$
\end{theorem}
\begin{proof}
By monotonicity, it suffices to prove the conjecture for rational $r,s$. So let $0 \le p_1/q_1<1$
and $0\le p_2/q_2<1$ be arbitrary. Let $W$ be a cyclic $XY$ word admissible for $q_1,q_2$ with 
$R(ab,p_1/q_1,p_2/q_2)=R(ab,p_1/q_1,p_2/q_2,W) =n/m$.
We assume $n/m$ is reduced, so that $ab$ has an orbit of period $m$ on $W$. We decompose $W$
into $m$ subwords $W=T_1T_2\cdots T_m$. For each $i$, let $T_i^+$ denote the rightmost
letter of $T_i$. Then $ab(T_i^+) = T_{n+i}^+$, indices taken mod $m$.

There is another (cyclic) decomposition of $W$ into $m$ subwords $U_1U_2\cdots U_m$
such that $b(T_i^+)=U_i^+$ and $a(U_i^+) = T_{n+i}^+$, where we denote the rightmost letter
of $U_i$ by $U_i^+$. Each $U_i^+$ is therefore a $Y$ and each $T_i^+$ is an $X$, so the
set of endpoints of these words are {\em disjoint}. Let $V_1V_2\cdots V_{2m}$ be the
common refinement.

We now show that we can adjust the letters in $W$ without affecting the dynamics. First of
all, any reordering of the letters in each $V_i$ which leaves the last letter intact will not
affect the dynamics, so without loss of generality we assume that $V_i$ is of the form
$X^xY^y$ if $V_i^+ = U_j^+$ for some $j$ (possibly with $x=0$), 
and $V_i = Y^yX^x$ if $V_i^+ = T_j^+$ for some $j$ (possibly with $y=0$).

Next, suppose some $T_i$ is entirely contained in some $U_j$, so that $T_i=V_k$ for some $k$,
and $T_i=Y^yX^x$. Notice $V_{k-1}$ is also entirely contained in $U_j$, and 
$V_{k-1}^+ = T_{i-1}^+$.
We claim that moving the (possibly empty) string $Y^y$ left to the start of $V_{k-1}$ will
not decrease the rotation number of $W$. First, we still have $a(U_l^+)=T_{n+l}^+$ for all $l$ 
(since $a$ ignores $Y$'s). Second, we have $b(T_l^+)\ge U_l^+$ for all $l$, since the
number of $Y$'s between $T_l^+$ and $U_l^+$ is either the same or is decreased by this
transformation. This proves the claim.

It follows that whenever we have a string of
consecutive $T_i$'s entirely contained in some $U_j$, we can move the $Y$'s out of the $T_i$'s
and to the left side of $U_j$.
After finitely many transformations of this kind, we can assume that every $U_j$ is of the form
$Y^{z_j}X^{x_j}Y^{y_j}$ (where {\it a priori} possibly $x_j$ and/or $z_j$ are $0$).

\medskip

Now, we claim that in fact every $x_j>0$. For, otherwise, there is some $U_j$ which consists entirely
of $Y$'s. But then $T_{n+j-1}^+=a(U_{j-1}^+)=a(U_j^+)=T_{n+j}^+$ which is absurd. The claim follows.
In particular, we can conclude that there is some $l$ so that 
$U_{i+l}^+ < T_{n+i}^+ < U_{i+l+1}^+$ for all $i$.

But from this the theorem follows easily. Since $a(U_i^+)=T_{n+i}^+$, from the definition of
the $a$ transformation we get an inequality $p_1+1 \ge x_i+x_{i+1} + \cdots + x_{i+l-1}+1$ and
therefore $p_1 \ge \sum_{j=0}^{l-1} x_{i+j}$.
Since this is true for every $i$, and since $\sum_{j=1}^m x_j = q_1$ we conclude $p_1/q_1 \ge l/m$.
Similarly, $p_2/q_2 \ge (n-l-1)/m$. But $R(ab,l/m,(n-l-1)/m,(XY)^m)=n/m$ and the theorem is proved.
\end{proof}

The proof of Theorem~\ref{ab_theorem} is expressed in terms of the purely combinatorial
question of maximizing $R(ab,p_1/q_1,p_2/q_2,W)$ over admissible $q_1,q_2$ words $W$.
It turns out that for the case of $ab$ there is a simple formula relating $R$ to $r$ which solves
the combinatorial question of minimizing $R(ab,p_1/q_1,p_2/q_2,W)$ over admissible 
$q_1,q_2$ words $W$. This is the following duality formula:

\begin{proposition}[Duality formula]\label{duality_formula}
Suppose $p_1/q_1$ and $p_2/q_2$ are reduced fractions. There is a formula
$$r(ab,p_1/q_1,p_2/q_2) + R(ab,(q_1-p_1-1)/q_1,(q_2-p_2-1)/q_2) = 2$$
\end{proposition}
\begin{proof}
Without loss of generality we can assume $0\le p_i/q_i < 1$. We have
$R(ab,(p_1-q_1)/q_1,(p_2-q_2)/q_2)=R(ab,p_1/q_1,p_2/q_2)-2$.
On the other hand, $R(ab,(p_1-q_1)/q_1,(p_2-q_2)/q_2)$ may be calculated as the maximum,
over all $W$ admissible for $q_1,q_2$ of $n/m$, where we alternate between moving {\em left}
$q_1-p_1$ $X$'s and $q_2-p_2$ $Y$'s. Changing the orientation of the circle, this evidently
computes $-r(ab,(q_1-p_1-1)/q_1,(q_2-p_2-1)/q_2)$.
\end{proof}

Figure~\ref{JN_ziggurat} shows the graph of $R(ab,\cdot,\cdot)$ over $[0,1]\times[0,1]$.
Discontinuities of $R$ are represented by vertical walls. Because of the monotonicity of
$R$, a camera situated at $(-1,-1,3)$ can see the entire graph without occlusion. Because
of its stepwise nature, we refer to such graphs as {\em ziggurats}.

\begin{figure}[htpb]
\labellist
\small\hair 2pt
\endlabellist
\centering
\includegraphics[scale=0.2]{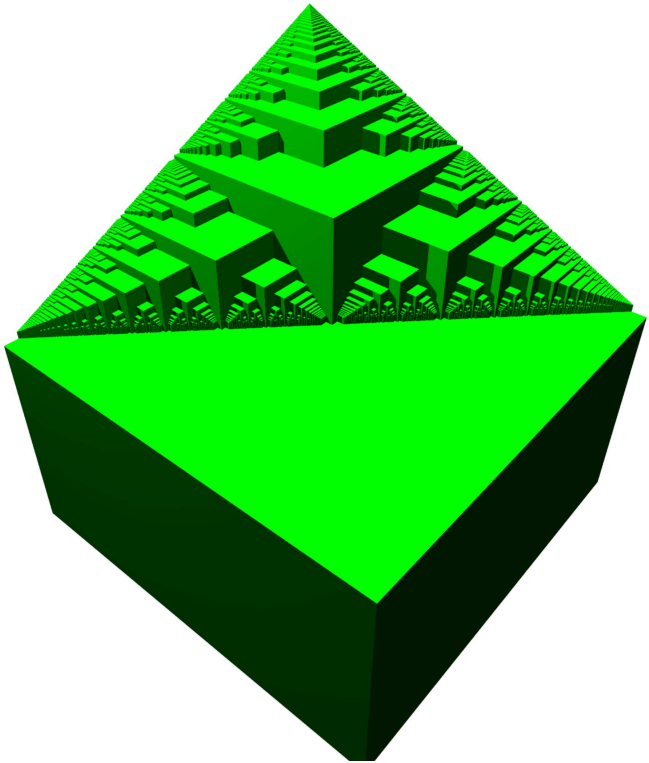}
\caption{The Jankins--Neumann ziggurat (i.e.\/ the graph of $R(ab,\cdot,\cdot)$).}\label{JN_ziggurat}
\end{figure}

\subsection{Stairsteps}

We have described an algorithm to compute $R$. However, this algorithm as stated is
very inefficient because of the very large number of admissible $q_1,q_2$ words for
large $q_i$. Given a positive word $w$ and a fixed $p_1/q_1$, the function
$t \to R(w,p_1/q_1,t)$ only takes on finitely many values in the interval $t\in [0,1)$
(for, it is {\it a priori} bounded, and rational with denominator $\le q_1$).
Moreover, we already know that this function is nondecreasing as a function of $t$,
and therefore it is completely specified if we know the finitely many values
that are achieved, and the infimal $t_i$ at which each value is achieved.

The following theorem says that these finitely many values are {\em rational}, that
each value is achieved at some {\em minimal} $t_i$, and gives an algorithm to compute
them.

\begin{theorem}[Stairstep Theorem]\label{stairstep_theorem}
Let $w$ be positive, and suppose we are given rational numbers $p/q$ and $c/d$ so that
$c/d$ is a value of $R(w,p/q,\cdot)$ (so necessarily $d \le q$). Then
$\inf \lbrace t : R(w,p/q,t)=c/d \rbrace$ is rational, 
and there is an algorithm to compute it. Moreover,
if $u/v$ is this infimal value, $R(w,p/q,u/v)=c/d$.
\end{theorem}
\begin{proof}
Given $w$, $p/q$ and $c/d$ we compute the infimal $t$ such that $R(w,p/q,t)\ge c/d$;
this will give us the same answer.
It suffices to compute the infimum over rational $t=u/v$; we give an algorithm to do this
whose output is evidently rational.
Without knowing $u$ and $v$ in advance, we let $W$ be an admissible $q,v$ word
for which $R(w,p/q,u/v)\ge c/d$ is achieved.
We write $W$ as $XY^{t_1}XY^{t_2}\cdots XY^{t_q}$ where the $t_i$ are (for the moment)
real variables subject to linear constraints $t_i\ge 0$ and $\sum t_i = v$.

We assume $c/d$ is reduced. Suppose equality is achieved, so that
$w$ has a periodic orbit of period $d$.
After replacing $w$ by a cyclic permutation we assume it begins with a string of $b$'s and
ends with a string of $a$'s, so that the periodic orbit begins in the terminal
string of $Y$'s. We measure the number of $X$'s we jump over or land on at each step. With each
maximal $a^m$ string we jump over precisely $mp+1$ $X$'s, landing on the last one. 
With each maximal $b^m$ string,
if we start at the $i$th $X$, we jump over $l$ $X$'s where $l$ is the smallest number
such that $t_i + t_{i+1} + \cdots + t_{i+l} \ge mu+1$. We can rewrite this condition by
saying that $l$ is the {\em biggest} number 
such that $t_i + t_{i+1} + \cdots + t_{i+l-1}\le mu$; the advantage of this reformulation
is that this latter inequality is {\em homogeneous}. Note that we must allow the
possibility $l=0$ (if $t_i > mu$), in which case this inequality is vacuous.
If we only have $R(w,p/q,u/v)\ge c/d$ then this inequality still holds, but we do not
assume $l$ is the biggest number with this property.

Write $w^d= b^{\beta_k} a^{\alpha_k} \cdots b^{\beta_1} a^{\alpha_1}$ (if $w$ is primitive,
the indices are periodic with period $k/d$). Then we will apply $k$ strings of $a$'s and $b$'s.
Let $l_i$ for $1\le i\le k$ be the value of $l$ as above when we apply the string
$b^{\beta_i}$. Then we obtain an equality $\sum_{i=1}^k (\alpha_ip+1+l_i) = cq$.
For each $i$ define $s_i=\sum_{j=1}^i (\alpha_jp+1) + \sum_{j=1}^{i-1} l_j$. Then our 
inequalities are precisely of the form
$t_{s_i} + t_{s_i+1} + \cdots + t_{s_i+l_i-1} \le \beta_iu$,
indices taken mod $q$.

Since our system of inequalities is homogeneous, linear and defined over $\Z$, 
we can find a solution in integers iff we can find a solution over the reals. We rescale
so that $v=1$. Our algorithm then has the following form:
First, enumerate all non-negative integral solutions to 
$\sum_{i=1}^k l_i = cq - \sum_{i=1}^k (\alpha_ip + 1)$ (there are only
finitely many such solutions, equal to the number of ordered partitions of
$cq - \sum_{i=1}^k (\alpha_ip + 1)$ into $k$ non-negative integers).
For each such solution, define $s_i$ by the formula 
$s_i=\sum_{j=1}^i (\alpha_jp+1) + \sum_{j=1}^{i-1} l_j$.
Then let $u$ be the smallest real number (necessarily rational) such that the
system of equations $\lbrace \sum_{i=1}^q t_i = 1$ and for each $1\le i \le k$,
$t_i \ge 0$ and
$t_{s_i} + t_{s_i+1} + \cdots + t_{s_i+l_i-1} \le \beta_i u \rbrace$ has a solution
(indices taken mod $q$).
The smallest $u$ over all such collections of $l_i$ is the desired quantity. Evidently
it is rational, and achieves the smallest possible value of $R(w,p/q,\cdot)$ which
is $\ge c/d$. If $c/d$ is achieved, $R(w,p/q,u)=c/d$.
\end{proof}

An interesting special case of the stairstep theorem is when $d=q$. 

\begin{proposition}\label{stairstep_lock}
Suppose there is some $t$ for which $R(w,p/q,t)=c/q$, where $p,c$ are coprime to $q$.
Let $w=a^{\alpha_1}b^{\beta_1}\cdots a^{\alpha_m}b^{\beta_m}$. Then 
$R(w,p/q,t)=c/q$ for $t$ on some interval $[u/nq,u/nq+\epsilon)$ where $u/n$ is the
least rational for which
$$c - m - h_a(w)p = \sum_{i=1}^m \left\lfloor \frac{\beta_iu+\epsilon} {n} \right\rfloor$$
\end{proposition}
\begin{proof}
Let $W=XY^{t_1}XY^{t_2}\cdots XY^{t_q}$ as in the proof of
Theorem~\ref{stairstep_theorem}. Since $w$ has a periodic orbit of period
exactly $q$, any $b$-string starting on adjacent $X$'s must land in adjacent
$Y^*$ strings. This dramatically cuts down on the number of partitions we need
to consider; for instance, the $l_i$ as in Theorem~\ref{stairstep_theorem} are
periodic with period $k/d$.

Because of this, the constraints for the linear programming problem
are invariant under cyclic permutation of the variables $t_i$, and by convexity,
setting all $t_i$ equal gives an extremal solution. The result follows
\end{proof}

Proposition~\ref{stairstep_lock} gives rise to the following inequality:

\begin{proposition}\label{lock_inequality}
For any positive word $w$ of the form $w=a^{\alpha_1}b^{\beta_1}\cdots a^{\alpha_m}b^{\beta_m}$,
if $R(w,p/q,t)=c/q$ there is an inequality
$$|c/q - h_a(w)p/q - h_b(w)t|\le 2m/q$$
\end{proposition}
\begin{proof}
We always have $R(w,p/q,t)\ge h_a(w)p/q + h_b(w)t$ coming from the representation
with $a=R_{p/q}$ and $b=R_t$, so we just need to prove the inequality for $t=u/nq$
as in Proposition~\ref{stairstep_lock}.

We have
$$c/q = m/q + h_a(w)p/q + \frac 1 q \sum_{i=1}^m \left\lfloor \frac{\beta_iu+\epsilon} {n} \right\rfloor$$
Since $\sum \beta_iu/qn = h_b(w)u/nq$ the inequality follows.
\end{proof}

If $q$ is big compared to $m$, Proposition~\ref{lock_inequality} 
says that $R(w,p/q,t)$ is very close to $h_a(w)p/q + h_b(w)t$, which 
is achieved for the linear representation $a=R_{p/q}$, $b=R_t$.
Combined with the inequality in Theorem~\ref{stability_theorem}, we obtain strong
constraints on the values of $R(w,p/q,\cdot)$ achieved, especially if $w$ contains
a substring of the form $a^m$ with $m$ large.

It is natural to wonder whether the inequality in Proposition~\ref{lock_inequality}
can be generalized to cases when the denominator of $R(w,r,s)$ is strictly smaller
than the denominators of $r$ and $s$; we return to this question in 
\S~\ref{slippery_subsection}.

\subsection{Speeding up the computation}

The algorithm described in the proof of Theorem~\ref{stairstep_theorem} has
one big bottleneck, namely the need to enumerate the partitions of
$cq - \sum_{i=1}^k (\alpha_ip + 1)$ into $k$ non-negative integers $l_i$. This
number of partitions is exponential in $k$, which is itself linear in $w$ and
$d\le q$. So the runtime of the algorithm above is {\it a priori} exponential in 
$w$ and $q$. 

Apart from
the inequalities $\sum_{i=1}^q t_i= 1$ and $t_i \ge 0$, we are left with the
``$s$-inequalities'' which are all of the form
$t_{s_i} + t_{s_i+1} + \cdots + t_{s_i+l_i-1} \le \beta_i u$.
For each $i$ with $1\le i\le k$ let $J_i$ denote the ``interval'' $[s_i,s_i+l_i-1]$ in
the ``circle'' $\Z/q\Z$. A partition $L:=l_1,l_2,\cdots,l_k$ of $cq - \sum_{i=1}^k (\alpha_ip + 1)$
thereby determines a combinatorial
configuration of intervals in a circle (weighted by integers $\beta_i$), 
and the minimal $u(L)$ depends only on the
combinatorics of this (weighted) configuration. 

In fact, we consider vectors of non-negative reals $r_i$ so
that the weighted sum of indicator functions $\chi(r):=\sum r_i \chi_{J_i}$ is $\ge 1$
everywhere in the circle. The $s$-inequalities imply $u\ge 1/(\sum r_i \beta_i)$, and 
linear programming duality gives the equality $u(L) = \sup_{\sum r_i\chi_{J_i}\ge 1} 1/(\sum r_i \beta_i)$. Hence the problem
becomes to minimize $\sum r_i \beta_i$ subject to $\sum r_i\chi_{J_i}\ge 1$; 
informally, to cover $S^1$ as ``efficiently as possible''
with the intervals $J_i$. If we call the minimum of $\sum r_i \beta_i$ the {\em efficiency}
of a covering, then since $u=\min_L u(L)$,
we seek the partition giving rise to the {\em covering of least efficiency}.

A {\em partial} partition is a vector $K:=l_1,l_2,\cdots,l_j$ with $j<k$ and 
$\sum_{i=1}^j l_i < cq - \sum_{i=1}^k (\alpha_ip + 1)$; such a partial partition can
be extended to a (complete) partition $L$ as above in potentially many ways; we write $K < L$ if $L$ 
extends $K$ to a complete partition. For a partial partition $K$, define 
$u(K)= \sup 1/(\sum r_i \beta_i)$ over all coverings of $S^1$ by intervals $J_i$ with $1\le i\le j$;
equivalently, we let $u(K)$ be the smallest $u$ for which there is a feasible solution to the
linear programming problem determined by the $s$-inequalities coming from $K$. Each
additional $s$-inequality can only reduce the space of feasible solutions, and therefore
$u(K) \le u(L)$ for every extension $K < L$. But $u=\min_L u(L)$. Consequently, if there
is a complete partition $L'$ with $u(K) > u(L')$, then $u(L) > u(L')$ for every $K< L$,
and therefore we can ignore all extensions $L$ of $K$ when computing $u$.

\begin{example}
Since we assume {\it a priori} that $u \le 1$, we must have $l_i \le q\beta_i$ for all $i$.
\end{example}

If we enumerate partitions lexicographically, inequalities as above let us prune the 
tree of partitions and speed up the computation of $u$. Figure~\ref{abaab_ziggurat} shows 
the ziggurat for $w=abaab$, which is computed using a combination of methods. 

\begin{figure}[htpb]
\labellist
\small\hair 2pt
\endlabellist
\centering
\includegraphics[scale=0.2]{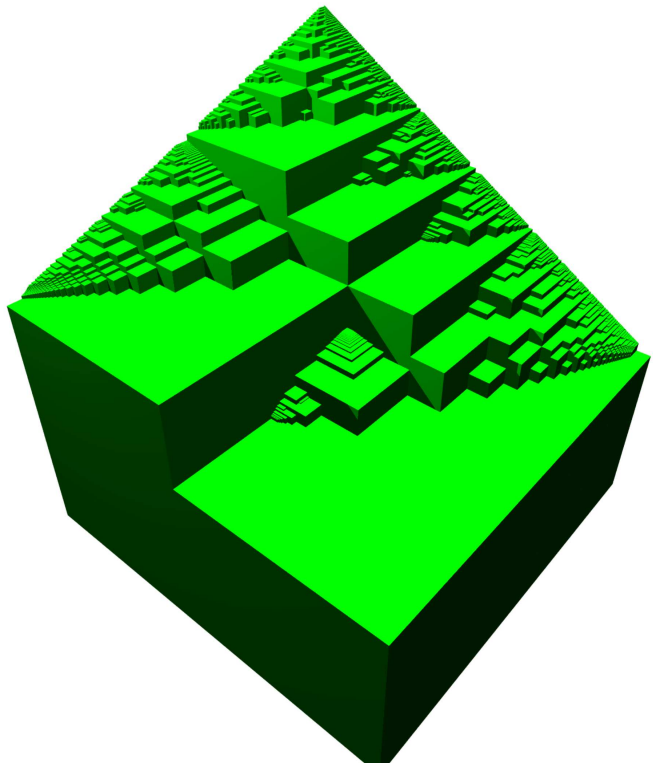}
\caption{The $abaab$ ziggurat.}\label{abaab_ziggurat}
\end{figure}

\subsection{Limits and rotations}

We introduce the notation $R(w,p_1/q_1,p_2/q_2-)$ to mean the
limit of $R(w,p_1/q_1,t)$ as $t \to p_2/q_2$ from below. Since
$R(w,p_1/q_1,t)$ is monotone nondecreasing as a function of $t$, and
takes only rational values with denominators bounded by $q_1$, this
limit is well-defined and rational, with denominator bounded by $q_1$,
and is achieved on a semi-open interval of values $[u/v,p_2/q_2)$
where $u/v$ can be determined by Theorem~\ref{stairstep_theorem}.
We similarly introduce the notation $R(w,p_1/q_1-,p_2/q_2)$ and $R(w,p_1/q_1-,p_2/q_2-)$.

\begin{proposition}\label{approach_from_below}
Let $w$ be positive, and suppose we are given rational numbers
$p_1/q_1$ and $p_2/q_2$. Then $R(w,p_1/q_1-,p_2/q_2)$ is the supremum of
$\rotz(w)$ for all representations with $a$ conjugate to the rotation
$R_{p_1/q_1}$ and $\rotz(b)=p_2/q_2$ (and similarly with the roles of $a$ and
$b$ interchanged), 
and $R(w,p_1/q_1-,p_2/q_2-)$ is the
supremum of $\rotz(w)$ for all representations with $a$ conjugate to
$R_{p_1/q_1}$ and $b$ conjugate to $R_{p_2/q_2}$.
\end{proposition}
\begin{proof}
We first claim that for any $a$ with
$\rotz(a)<p_1/q_1$ there is some $\alpha$ conjugate to $R_{p_1/q_1}$
satisfying $\alpha(p)\ge a(p)$ for all $p$. As in the proof
of Lemma~\ref{smooth_approximation} we can produce a smooth
family $a_t$ interpolating between $a^-$ and $a^+$, all $C^0$ close to
$a$. Some $a_t$ with $a_t(p) \ge a(p)$ has irrational rotation number, 
and is therefore conjugate (by some $g$)
to some $R_\theta$ with $\rotz(a)\le \theta < p_1/q_1$.
If we let $\alpha$ be obtained by conjugating $R_{p_1/q_1}$ by $g$
then $\alpha$ satisfies the desired properties.

Conversely, any homeomorphism conjugate to $R_{p_1/q_1}$ can be $C^0$
approximated by homeomorphisms conjugate to $R_t$ for any $t<p_1/q_1$.
\end{proof}

An important special case is $1-$. Because $R_1$ is central,
for any positive $w$ we have $R(w,1-,p/q)=h_a(w)+h_b(w)p/q$ and
$R(w,p/q,1-) = h_a(w)p/q+h_b(w)$. This gives rise to straight lines of slope $1$,
clearly visible in the graphs of 
$R(ab,\cdot,\cdot)$ and $R(abaab,\cdot,\cdot)$ 
in Figures~\ref{JN_ziggurat} and \ref{abaab_ziggurat}.

The Stairstep Theorem (i.e.\/ Theorem~\ref{stairstep_theorem}) implies
that for each $w$ and each $p/q$ there is some smallest $u/v$ with 
$R(w,p/q,u/v) = h_a(w)p/q+h_b(w)$ (and similarly with the roles of 
$a$ and $b$ interchanged). Proposition~\ref{stairstep_lock} says that
$u/v=u/nq$ where $u/n$ is minimized subject to 
$$h_b(w)q - m = \sum_{i=1}^m \left\lfloor \frac{\beta_iu+\epsilon} {n} \right\rfloor$$
Note that the result does not depend on $p$; this ``explains'' why the fringes
of the ziggurats appear periodic on every scale near the sloped edges. 
Solving for integers $u$ and $n$  to minimize $u/n$ is straightforward.

\begin{example}
For $w=ab$, $R(ab,p/q,t)=1+p/q$ on the maximal interval $t=[(q-1)/q,1)$ (this
follows from Theorem~\ref{ab_theorem}, of course).
\end{example}

\begin{example}
For $w=abaab$, $R(abaab,p/q,t)=2+3p/q$ for $3$ coprime to $q$ on the maximal interval
$t=[(q-1)/q,1)$, and $R(abaab,t,p/q)=3+2p/q$ for odd $q$ on the maximal interval
$t=[(2q-1)/2q,1)$.

Without the condition that $h_a(w)$ and $q$ are coprime, the formula is more tricky;
for example, $R(abaab,2/3,t)=4$ on the maximal interval $t=[1/2,1)$, and 
$R(abaab,5/6,t)=9/2$ on the maximal interval $t=[3/4,1)$.
\end{example}

\begin{example}\label{half_abaab_example}
We have $R(abaab,1/2-,t) = 1+R(abb,1/2-,t) = 1+R(ab,1/2-,2t)$ which is equal to
$2+1/(2p+1)$ for $2t\in [(p+1)/(2p+1),p/(2p-1))$ for $p>1$, and $2+p/(2p+1)$ for
$2t\in [2p/(2p+1),(2p+2)/(2p+3))$.

On the other hand, $R(abaab,t,1/2-) = 1 + R(abaab^{-1},t,1/2-)$ where the notation
here simply means representations for which $b$ is conjugate to $R_{1/2}$.
Obviously $R(abaab^{-1},t,1/2-) \le R(ab,t,2t)$. On the other hand, by the proof of
Theorem~\ref{ab_theorem}, we know that for $t=p/q$ with $q$ odd, the extremal $q,q$
word for $R(ab,p/q,2p/q)$ is $(XY)^q$; evidently the dynamics of $a^2$ and $b$ are conjugate
in this case by an element which is itself conjugate to $R_{1/2}$, so we obtain an
equality $R(abaab,p/q,1/2-) = 1 + R(ab,p/q,2p/q)$ for $q$ odd (and in fact
by the same reasoning, 
$R(aba^nb,p/q,1/2-) = 1 + R(ab,p/q,cp/q)$ for $q$ odd and $n$ coprime to $q$).
\end{example}

\subsection{Slippery points}\label{slippery_subsection}

\begin{definition}
Let $w$ be positive. A pair of rational numbers $(r,s)$ is {\em slippery} for $w$ if 
there is a {\em strict} inequality $R(w,r',s') < R(w,r-,s-)$ for all $r'<r$, $s'<s$.
\end{definition}

If $(r,s)$ is not slippery, then there are $r' <r$ and  $s'<s$ with $R(w,r',s')=R(w,r-,s-)$.
It follows that this common value is achieved on the entire region $[r',r)\times [s',s)$
and is therefore rational.

\begin{example}
$(1,t)$ and $(t,1)$ are slippery for all positive words. $(1/2,1/2)$ is slippery for $abaab$,
by Example~\ref{half_abaab_example}. Experimentally, $(1/2,1/2)$ is slippery for
many other words; e.g. $abaababb$, $abaabaaaabb$.
\end{example}

The following conjecture generalizes Proposition~\ref{lock_inequality}, 
and is supported by some experimental evidence.

\begin{slippery_conjecture}\label{generalized_lock_inequality}
For any positive $w$ of the form $w=a^{\alpha_1}b^{\beta_1}\cdots a^{\alpha_m}b^{\beta_m}$, 
if $R(w,r,s)=p/q$ where $p/q$ is reduced, then $|p/q - h_a(w)r - h_b(w)s| \le m/q$.
\end{slippery_conjecture}

Note that for the ``linear''
representation in which $a=R_r$ and $b=R_s$, we have $\rotz(w)=h_a(w)r+h_b(w)s$. So this
conjecture says that the bigger the denominator of $R(w,r,s)$, the closer the extremal
representation (one realizing $\rotz(w)$) must be to the linear representation.
The idea behind this conjecture is that the ``more nonlinear'' an extremal representation
is, the more rigid it is, and the smaller the denominator of
$R(w,r,s)$. Contrapositively, the bigger the denominator, the less rigid (and hence the more
slippery) and the closer to linear. This is only heuristic reasoning (and hand-wavy at that),
but it is in sympathy with the Stability Theorem.

This conjecture is very interesting in view of the following consequence:

\begin{proposition}
The Slippery Conjecture
implies that for any slippery $(r,s)$, we have $R(w,r-,s-)=h_a(w)r+h_b(w)s$.
In particular, it implies that $R(w,r-,s-)$ is rational for all rational $(r,s)$.
\end{proposition}
\begin{proof}
If $(r,s)$ is slippery, there are $r'<r$, $s'<s$ arbitrarily close to $r$, $s$ for which
the denominator of $R(w,r',s')$ is arbitrarily large. Since $R(w,r-,s-)$ is
rational for $(r,s)$ rational and not slippery, the proposition follows.
\end{proof}

The Slippery Conjecture
has been experimentally checked for several words, up through the range $q\le 14$,
which is close to the limit of what can be computed easily.
Figures~\ref{QePlot1} and \ref{QePlot2} show the range of experimentally computed
values of $|p/q - h_a(w)r - h_b(w)s|$ for
words $abaab$ and $abaabbabbbababaab$ respectively, for $q\le 14$. In each case,
the maximal error $m/q$ is precisely achieved for ``most'' $q$.

\begin{figure}[htpb]
\labellist
\small\hair 2pt
\endlabellist
\centering
\includegraphics[scale=0.5]{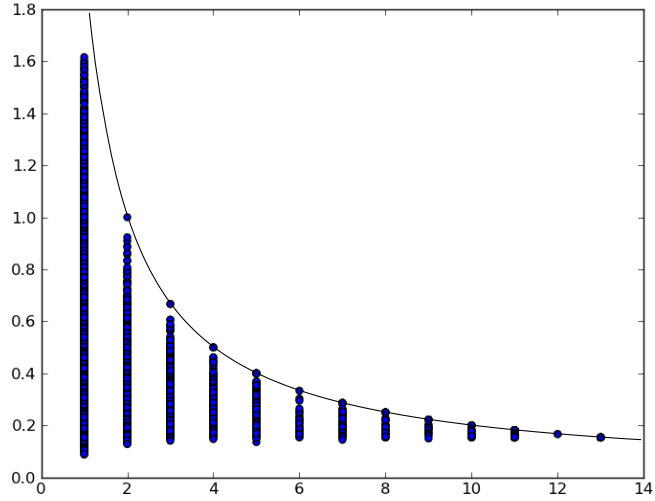}
\caption{Scatter plot of $|R(w,r,s) - h_a(w)r - h_b(w)s|$ versus $q$ for $w=abaab$ for 
all rational values of $r$ and $s$ with denominator at most $13$.  The graph $m/q = 2/q$ is 
also pictured.}\label{QePlot1}
\end{figure}

Note that for any $w$ we have $R(w,1/2-,1/2-)=(h_a(w)+h_b(w))/2$. For, any representation in which
both $a$ and $b$ are conjugate to rigid rotations with rotation number $1/2$ factors through
the infinite dihedral group, which is amenable. On any representation which factors through
an amenable group, rotation number becomes a {\em homomorphism}. This
``explains'' why $(1/2,1/2)$ is likely to be slippery for many $w$.

\begin{figure}[htpb]
\labellist
\small\hair 2pt
\endlabellist
\centering
\includegraphics[scale=0.5]{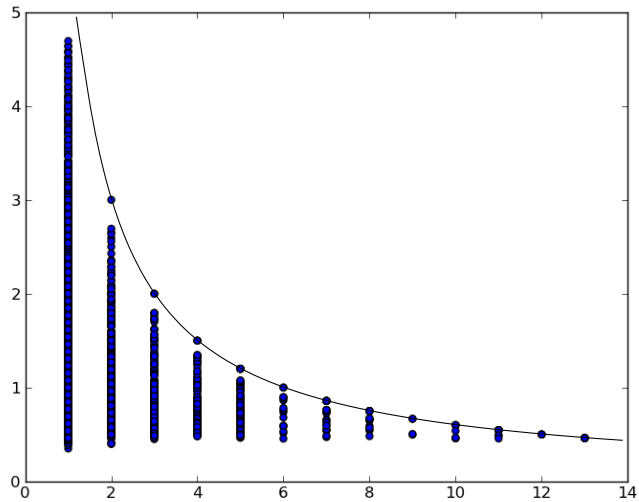}
\caption{Scatter plot of $|R(w,r,s) - h_a(w)r - h_b(w)s|$ versus $q$ and $m/q=6/q$  
for $w=abaabbabbbababaab$.}\label{QePlot2}
\end{figure}

\subsection{Immersions and $\scl$}\label{immersion_subsection}

We now describe an unexpected connection to hyperbolic geometry and stable
commutator length. For the benefit of the reader, we quickly recall some basic elements
of the theory of stable commutator length (for a more substantial introduction 
see \cite{Calegari_scl}).

\begin{definition}\label{scl_definition}
If $G$ is a group, and $g\in[G,G]$, the commutator length $\cl(g)$ is the least number of
commutators in $G$ whose product is $g$, and the {\em stable commutator length} is the
limit $\scl(g):=\lim_{n\to \infty} \cl(g^n)/n$. If $g_1,g_2,\cdots,g_m$ is a collection
of elements with $\prod_i g_i \in [G,G]$ then 
$\scl\left(\sum g_i\right):=\lim_{n\to\infty} \cl\left(\prod_i g_i^n\right)/n$.
\end{definition}

\begin{definition}
If $G$ is a group, a {\em homogeneous quasimorphism} is a function $\phi:G \to \R$ satisfying
$\phi(g^n)=n\phi(g)$ for all $g\in G$ and $n\in\Z$, and for which there is a least non-negative
real number $D(\phi)$ (called the {\em defect}) such that $|\phi(gh)-\phi(g)-\phi(h)|\le D(\phi)$
for all $g,h\in G$.
\end{definition}

Generalized Bavard Duality (\cite{Calegari_scl}, Thm.~2.79) says that
for any $G$, and for any finite set of elements $g_i$ with $\prod g_i \in [G,G]$, there is an equality
$$\scl\Bigl(\sum_i g_i\Bigr) = \sup_\phi \left(\sum \phi(g_i)\right)/2D(\phi)$$ where the supremum is taken over
all homogeneous quasimorphisms.

Any representation $\rho:G \to \homeo^+(S^1)~$ gives rise to a homogeneous quasimorphism
on $G$, namely rotation number. In general, this homogeneous quasimorphism satisfies $D(\phi)\le 1$,
where generically $D(\phi)=1$. Therefore stable commutator length can be used to give an
upper bound on $R$. However what is quite surprising is that this upper bound is often
{\em sharp}, under geometrically meaningful conditions. This is the content of the next theorem.

\begin{theorem}\label{scl_rotation_formula}
Let $\gamma_w$ be the unique geodesic representative of the word $w$ in the hyperbolic
$(q_1,q_2,\infty)$-orbifold $O(q_1,q_2)$. If $\gamma_w$ virtually bounds an
positively immersed subsurface $S$ in
$O(q_1,q_2)$ then $$R(w,1/q_1-,1/q_2-)=h_a(w)/q_1+h_b(w)/q_2+\area(S)/2\pi$$ 
In general there are
inequalities $$R(w,1/q_1-,1/q_2-)\ge h_a(w)/q_1+h_b(w)/q_2+A(\gamma_w)/2\pi$$ where
$A(\gamma_w)$ is the algebraic area enclosed by $\gamma_w$, and
$$R(w,1/q_1-,1/q_2-)\le h_a(w)/q_1+h_b(w)/q_2+2\;\scl_{G(q_1,q_2)}(w)$$ where $\scl$ (denoting
stable commutator length) is computed in the group 
$G(q_1,q_2):=\langle a,b \;|\; a^{q_1}=b^{q_2}=1\rangle$.
\end{theorem}
\begin{proof}
Let $H(q_1,q_2)$ be the central extension of $G(q_1,q_2)$, defined by the presentation
$H(q_1,q_2):=\langle a,b,t \;|\; [t,a]=[t,b]=1,a^{q_1}=t,b^{q_2}=t\rangle$. $G(q_1,q_2)$ is the
(orbifold) fundamental group of the $(q_1,q_2,\infty)$-orbifold $O(q_1,q_2)$, and $H(q_1,q_2)$ is
the fundamental group of its (orbifold) unit tangent bundle.

The (unique complete finite area)
hyperbolic structure on $O(q_1,q_2)$ gives rise to a representation $G \to \PSL(2,\R)$ which
can be thought of as a subgroup of $\homeo^+(S^1)$. This
is covered by a representation $H \to \homeo^+(S^1)^\sim$. The elements $a$ and $b$
are conjugate to rigid rotations, and for this representation
$\rotz(w)-h_a(w)/q_1-h_b(w)/q_2$ is equal to $A(\gamma_w)/2\pi$ (this essentially follows
from the Gauss-Bonnet theorem). This gives one inequality.

Generalized Bavard duality (applied to any representation)
and the fact that the defect of any rotation quasimorphism is at most $1$, 
gives the upper bound $R(w,1/q_1-,1/q_2-)-h_a(w)/q_1-h_b(w)/q_2 \le 2\;\scl_H(w-h_a(w)a-h_b(w)a)$.
If $\gamma_w$ virtually bounds a positively immersed surface $S$ in $O$, this surface and the
rotation quasimorphism are both extremal for $w$ (see \cite{Calegari_faces} for details). 
In this case, both inequalities are equalities (with $\area(S)=A(\gamma_w)$).

Since $H$ is an amenable extension of $G$, and $H^2(G;\R)=0$, the projection from $H$ to
$G$ is an isometry for $\scl$, so $\scl_H(w-h_a(w)a-h_b(w)b)=\scl_G(w-h_a(w)a-h_b(w)b)$. But
in $G$, the elements $a$ and $b$ have finite order, so this last term is equal to $\scl_G(w)$.
\end{proof}

\begin{remark}
The group $G(q_1,q_2)$ is virtually free, and therefore $\scl$ is rational and can be computed
in polynomial time (for fixed $q_1,q_2$) by the method of \cite{Calegari_rational} 
(also see \cite{Calegari_scl}, Chapter 4).
In fact, an algorithm due to Walker lets one compute $\scl$ in polynomial time in word length
{\em and in $q_1,q_2$}.

Moreover, for fixed $q_1,q_2$, the function $w \to A(\gamma_w)$ is an example of a 
{\em bicombable function} (see \cite{Calegari_Fujiwara}), and can be computed in
{\em linear} time. This makes it practical to actually compute the bounds 
in Theorem~\ref{scl_rotation_formula}.
\end{remark}

\begin{remark}
The question of which geodesics virtually bound immersed subsurfaces of hyperbolic surfaces
is very subtle and difficult. This question is pursued in some detail in
\cite{Calegari_faces}, and in the special case of the modular orbifold in \cite{Calegari_Louwsma}.
\end{remark}

\begin{example}
The special case $w=ab$ is particularly simple. In this case,
$\gamma_{ab}$ does not have a geodesic representative and one must work with the cusp
representative instead, which is the boundary of the $(q_1,q_2,\infty)$ orbifold, and
therefore always tautologically bounds $O(q_1,q_2)$. Hence
$R(ab,1/q_1-,1/q_2-)=1/q_1+1/q_2+\area(O(q_1,q_2))/2\pi=1$.
\end{example}

In order to apply Theorem~\ref{scl_rotation_formula} for general 
$R(w,p_1/q_1-,p_2/q_2-)$ we use the following fact. Let $\phi_{p_1,p_2}:F \to F$ be defined
on generators by $a\to a^{p_1}$ and $b\to b^{p_2}$. Then
$R(\phi_{p_1,p_2}(w),1/q_1,1/q_2) = R(w,p_1/q_1,p_2/q_2)$ and similarly with $p_i/q_i$ replaced by
$p_i/q_i-$. 

\begin{example}[$a^{p_1}b^{p_2}$]
$R(ab,p_1/q_1-,p_2/q_2-) = R(a^{p_1}b^{p_2},1/q_1-,1/q_2-)$. Let $G=G(q_1,q_2)$ as above.
Now, $\scl_G(a^{p_1}b^{p_2})=\scl_G(ab)=1/2-1/2q_1-1/2q_2$ so we get the inequality
$R(ab,p_1/q_1-,p_2/q_2-) \le 1 + (p_1-1)/q_1 + (p_2-1)/q_2$.
This is never sharp if at least one $p_i>1$, and therefore {\em none} of the
geodesics $\gamma_{a^{p_1}b^{p_2}}$ virtually bound a positively 
immersed surface in $O(q_1,q_2)$ except for $\gamma_{ab}$.
\end{example}

\subsection{Isobars}

In this section we prove a structure theorem for the level sets of $R$.
But first we must prove a couple of lemmas.

\begin{lemma}\label{feasible_rots}
Let $I_1,I_2,\cdots,I_k$ be a finite set of closed intervals in $S^1$,
and for each $i$ let $I_i^-$ be the initial point and $I_i^+$ the final
point of the interval (with respect to the orientation on $S^1$).
The set of values of $s$ for which there is a homeomorphism $a(\cdot)$ conjugate to
a  rotation $R_s$, and such that $a(I_i^-) = I_i^+$ for all $i$,
is a connected interval (possibly open or half-open) 
with rational endpoints. Moreover, the endpoints depend only
on the combinatorics of the $I_j$, and can be computed by linear programming.
\end{lemma}
\begin{proof}
A rotation number $s$ achieved by some $a$ as above is said to be {\em feasible}
for the collection of intervals. It is obvious from the definition that
the set of feasible $s$ for a given collection of intervals depends only on the
combinatorics.

The $I_j$ and their complements partition $S^1$ into disjoint intervals $J_i$. We
assign a variable $t_i$ to each $J_i$. Showing $s$ is feasible is equivalent to
the feasibility of the linear programming problem defined by the following
constraints:
\begin{enumerate}
\item{each $t_i$ is in $(0,1)$}
\item{for each $I_j$ the set of $t_i$ for which the $J_i \subset I_j$ sums to $s$}
\item{the sum of all the $t_i$ is $1$}
\end{enumerate}
The connectedness of the set of feasible $s$ follows from convexity. Rationality
follows from the form of the linear programming problem. The proof follows.
\end{proof}

\begin{remark}
The set of $s$ feasible for a given finite collection of intervals could easily
be empty; for example, if one interval is properly contained in another.
\end{remark}

Recalling the definition of $X(w,r,s)$ as in Lemma~\ref{interchange}, we
introduce the notation $X(w:t)$ for the set of $r,s$ for which $t\in X(w,r,s)$.
Evidently $X(w:t)$ is the intersection of the set of $r,s$ with $R(w,r,s)\ge t$
with the set of $r,s$ with $R(w,-r,-s)\ge -t$.

\begin{lemma}\label{rotations_dense}
Let $w$ be positive. $X(w:t)$ is the closure of the set of pairs $(r,s)$ such that there
is a representation with $a$ conjugate to a rotation $R_r$ and 
$b$ conjugate to a rotation $R_s$, and with $\rotz(w)=t$.
\end{lemma}
\begin{proof}
Suppose $(r,s)\in X(w:t)$. As $u$ varies in $(-\epsilon,\epsilon)$, the maximum
of $X(w,r+u,s+u)$ is continuous from the right, and the minimum is continuous
from the left. In order to complete the proof we make two observations.

On the one hand, by Lemma~\ref{smooth_family},
for $(r,s)$ irrational Herman numbers, any $t$ in the interior of $X(w,r,s)$ 
can be achieved by a smooth representation for which $a$ is conjugate to $R_r$
and $b$ is conjugate to $R_s$.

On the other hand, taking $a=R_r$ and $b=R_s$ on the nose,
we see that $h_a(w)r+h_b(w)s \in X(w,r,s)$ for all $(r,s)$. 

So we have two possibilities: either there is $u$ 
arbitrarily close to $0$ with $(r+u,s+u)$ both Herman irrationals and
$t$ in the interior of $X(w,r+u,s+u)$,
or else $h_a(w)r+h_b(w)s=t$; in either case we are done by one of the
two observations above.
\end{proof}

We are now ready to prove the main result of this section.

\begin{theorem}[Isobar Theorem]\label{isobar_theorem}
Let $w$ be positive. For any rational $p/q$ the set of $r,s\subset [0,1]\times[0,1]$ 
such that $R(w,r,s)\ge p/q$ is a finite sided rational polyhedron, whose boundary 
consists of finitely many horizontal or vertical segments.
\end{theorem}

We call the frontier of $R(w,r,s)\ge p/q$ the $p/q$ {\em isobar}. This is part 
of the frontier of the level set $R(w,r,s)=p/q$. Note that it is {\em not} true 
in general that the closure of the level set itself is a finite sided polyhedron, 
even for $w=ab$. Thus the set of
$r,s$ such that $R(w,r,s)\le p/q$ can be in general quite complicated.
For example, the level set $R(ab,\cdot,\cdot)=1$ has the line $r+s=1$ in its frontier,
whereas the level set $R(abaab,\cdot,\cdot)=2$ has infinitely many segments in its frontier.
See Figures~\ref{JN_ziggurat} and \ref{abaab_ziggurat}.

\begin{proof}
By the discussion above, it is sufficient to prove that $X(w:p/q)$ is a finite
sided rational polygon whose boundary consists of horizontal and vertical segments.
By Lemma~\ref{rotations_dense} we need to compute the set of $r,s$ for which
there are representations with $a$ and $b$ conjugate to rotations through
$r$ and $s$ respectively, and $\rotz(w)=p/q$. So consider such a representation.

Let $w=a^{\alpha_1}b^{\beta_1}\cdots a^{\alpha_m}b^{\beta_m}$
and let $M=\sum \alpha_i + \sum \beta_i$ be the word length of $w$.
By the hypothesis on the rotation number, we can tile $[0,p]$ by $qM$ intervals,
each of which is of the form $[t,a(t)]$ or $[t,b(t)]$. Consider the projection
of these $qM$ intervals to the circle; there are a large but finite number
of combinatorial types for the image; call such a combinatorial type a
{\em partition}. For each partition, we compute the set of $r$ feasible for the
$a$-intervals and the set of $s$ feasible for the $b$-intervals, by
Lemma~\ref{feasible_rots}; the set of $(r,s)$ compatible with both is therefore
a rectangle with rational vertices. The union of the interiors of these rectangles
over all partitions is dense in $X(w:p/q)$; the theorem follows.
\end{proof}

Figure~\ref{isobar_abaab_2_to_3} depicts $X(abaab:p/q)$ for a few simple values of $p/q$.

\begin{figure}[htpb]
\labellist
\small\hair 2pt
\endlabellist
\centering
\includegraphics[scale=0.4]{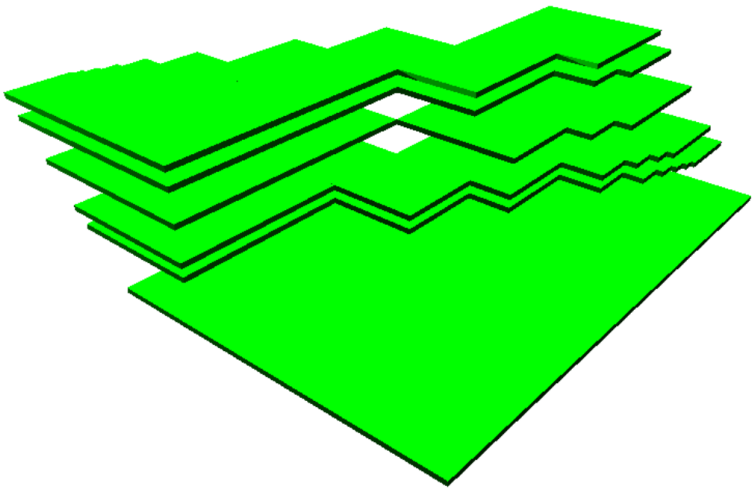}
\caption{Slices $X(abaab:p/q)$ for $p/q=2$, $9/4$, $7/3$, $5/2$, $8/3$, $11/4$.
The interior of the $5/2$ slice is disconnected at a slippery point.}\label{isobar_abaab_2_to_3}
\end{figure}

\section{Arbitrary words}\label{arbitrary_word_section}

\subsection{Semipositive words}

\begin{definition}
A word $w$ is {\em semipositive} if it either contains no $a^{-1}$ or no $b^{-1}$.
\end{definition}

Theorem~\ref{positive_rational} generalizes in a straightforward way to semipositive words.

\begin{theorem}\label{semipositive_rational}
Let $w$ be semipositive (without loss of generality, suppose it contains no $a^{-1}$).
If $r$ is rational, so is $R(w,r,s)$. Moreover, the denominator of $R(w,r,s)$ is
no bigger than the denominator of $r$.
\end{theorem}
\begin{proof}
After replacing $w$ by a cyclic conjugate, we can assume it ends with $a$.
Let $r=p_1/q_1$, and let $\Sigma_x=\cup_i x_i$ be a periodic orbit 
for $a$, so that $a(x_i)=x_{i+p_1}$.
Let $\alpha^+$ be defined by $\alpha^+(\theta)=a(x_{i+1})$ for
$\theta \in (x_i,x_{i+1}]$, and let $w^+$ be obtained by replacing $a$ with $\alpha^+$.
Then as in Theorem~\ref{positive_rational} $w$ and $w^+$ have the same rotation number,
whereas $w^+$ takes $\Sigma_x$ to itself, and therefore has a periodic orbit
with period $\le q_1$.
\end{proof}

\subsection{Rationality Conjectures}

In this section we state three conjectures on the rationality of $R$. These conjectures
are related, and we explain how Conjecture~\ref{rationality_conjecture} at least 
would follow if a certain dynamical problem 
(the {\em interval game}; see \S~\ref{interval_game_subsection}) always had a positive solution.

Unfortunately, there are instances of the interval game which are unwinnable (i.e.\/
unsolvable); however, it turns out that this dynamical problem {\em generically} 
has a positive solution, and that the exceptions must be quite special. 
This suggests a program to attack the rationality conjectures.

We adopt the following notational convention, which is consistent with our earlier use
for positive $w$:

\begin{definition}
For arbitrary $w$ and for real $r,s$, let $R(w,r-,s-)$ denote the supremum of
$\rotz(w)$ under all representations for which $a$ and $b$ are conjugate to
rotations $R_r$ and $R_s$ respectively.
\end{definition}

We would like to prove the following conjectures:

\begin{conjecture}\label{rigid_rationality_conjecture}
Let $w$ be arbitrary, and let $r,s\in\Q$. Then $R(w,r-,s-)\in\Q$.
\end{conjecture}

\begin{conjecture}\label{rationality_conjecture}
Let $w$ be arbitrary, and let $r,s \in \Q$. Then $R(w,r,s)\in \Q$.
\end{conjecture}

\begin{conjecture}\label{rational_maximum_conjecture}
Let $w\in [F,F]$. Then $\max_{r,s} R(w,r,s) \in \Q$.
\end{conjecture}

Although on the surface, 
Conjectures~\ref{rigid_rationality_conjecture} and \ref{rationality_conjecture}
seem very similar, the former very quickly reduces to the case of positive $w$:

\begin{proposition}\label{reduce_to_positive}
Conjecture~\ref{rigid_rationality_conjecture} is true if it is true for all positive
$w$.
\end{proposition}
\begin{proof}
Actually, the proof is a trick. Let $r=p_1/q_1$ and $s=p_2/q_2$.
Then $a^{q_1}$ is conjugate to the central element $R_{p_1}$,  
and similarly for $b^{q_2}$. It follows that for any factorization
$w=w_1w_2$ we have $w=w_1w_2 = w_1a^{q_1}w_2z^{-p_1}$, and therefore we can write
$w=w'z^{-N}$ for some sufficiently large integer $N$,
where $w'$ is positive, and $R(w,p_1/q_1-,p_2/q_2-) = R(w',p_1/q_1-,p_2/q_2-)-N$.
\end{proof}

It follows from Proposition~\ref{reduce_to_positive} that $R(w,r-,s-)$ 
can be computed, at least for $r,s\in\Q$, by the
method of \S~\ref{immersion_subsection}. 

There are at least two classes of $(r,s)$ for which $R(w,r-,s-)$ is known to be rational:
\begin{enumerate}
\item{at values of $(r,s)$ where $R(w,\cdot,\cdot)$ is
locally constant, we have $R(w,r-,s-)=R(w,r,s)\in\Q$; and}
\item{at values of $(r,s)$ where the upper bound in Theorem~\ref{scl_rotation_formula} is realized,
we have $R(w,r-,s-)\in\Q$.}
\end{enumerate}
It is nevertheless true that $R(w,r-,s-)$ 
can be quite complicated. 

\begin{example}\label{commutator_rotation}
For any $r,s$ (not necessarily rational) we have $R(aba^{-1}b^{-1},r-,s-)=0$. To see this, observe
that $ba^{-1}b^{-1}$ is conjugate to $R_{-r}$, and therefore there is some
point $p$ for which $ba^{-1}b^{-1}(p)=p-r$; but then $p$ is fixed by $aba^{-1}b^{-1}$, 
which therefore has rotation number $0$.
\end{example}

\begin{example}\label{commutator_all}
We now discuss $R(aba^{-1}b^{-1},r,s)$, as an interesting counterpoint 
to Example~\ref{commutator_rotation}. First of all, we claim that $R(aba^{-1}b^{-1},r,s)=0$
whenever $r$ or $s$ is irrational. We argue analogously to the case of Example~\ref{commutator_rotation}:
suppose $r$ is irrational, and let $\mu$ be an invariant probability measure for $a$
supported on an exceptional minimal set. If there is an interval $I$ with $\mu(I)=r$ and
$\mu(b^{-1}(I))<r$ then $aba^{-1}b^{-1}(I^-)<I^-$, so $\rotz(aba^{-1}b^{-1})\le 0$.
But if $\mu(b^{-1}(I))\ge \mu(I)$ for every interval $I$ with $\mu(I)=r$ then $b$ 
actually preserves $\mu$, and therefore the action is semiconjugate to a linear action, and $\rotz(aba^{-1}b^{-1})=0$.

If $r=p/q$ is rational, then $a(\cdot)$ has a periodic orbit
$x_1,x_2,\cdots,x_q$ with indices corresponding to the cyclic order, and
$a(x_i)=x_{i+p}$. If for some $i$, 
$x_j< b^{-1}(x_i) < x_{j+1}$ and $x_k < b^{-1}(x_{i-p}) < x_{k+1}$
with $k\le j-p-1$ then $aba^{-1}b^{-1}(x_i) < x_i$ and $\rotz(aba^{-1}b^{-1})\le 0$.
Otherwise we must have $k=j-p$, and $aba^{-1}b^{-1}(x_i) < x_{i+1}$, so that
$\rotz(aba^{-1}b^{-1})\le 1/q$. If $s=p'/q$ we can build an action which is a $q$-fold
cyclic cover of an action for which both $a$ and $b$ have fixed points; this shows
that $R(aba^{-1}b^{-1},p/q,p'/q)\ge 1/q$ and therefore $R(aba^{-1}b^{-1},p/q,p'/q)=1/q$.

Compare with \cite{Thurston_circles}, Remark~3.8.
\end{example}

\begin{figure}[htpb]
\labellist
\small\hair 2pt
\endlabellist
\centering
\includegraphics[scale=0.2]{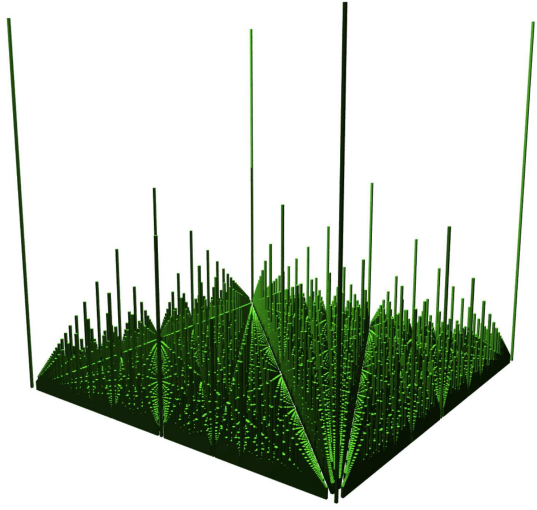}
\caption{The ``ziggurat'' for $aba^{-1}b^{-1}$.}\label{abAB_ziggurat}
\end{figure}

Note that Conjecture~\ref{rigid_rationality_conjecture} can only fail for some
$w,r,s$ if $(r,s)$ is slippery for $w'$ (with notation from the proof of 
Proposition~\ref{reduce_to_positive}). Therefore Conjecture~\ref{rigid_rationality_conjecture}
is implied by the Slippery Conjecture.

\subsection{The Interval Game}\label{interval_game_subsection}

We pursue the following strategy to attack Conjecture~\ref{rational_maximum_conjecture},
based more or less on the method of perturbation. This leads to a 
dynamical problem that we call {\em the interval game}. The structure of the set of 
``losing games'' is very interesting, even when restricted to a very simple class of games
(e.g. consisting entirely of rigid rotations). 

The strategy is as follows. Given a word $w$, and given $r,s\in\Q$ we
suppose that we have a representation for which $a$ and $b$ have the prescribed rotation
numbers, and for which $\rotz(w)$ is maximal. Suppose that $\rotz(w)$ is irrational.
We would like to adjust $a$ to a new map $a'$ (with the same rotation number as $a$)
which will adjust $w$ to a new $w'$ that
has a strictly bigger rational rotation number. If we could do this, we would
obtain a contradiction, and therefore conclude that $R(w,r,s)$ was rational after all.

We look for a suitable interval $I\subset S^1$ and require $a'$ to agree 
with $a$ outside $I$, but $a' > a$ on $I$. Providing $I$ can be chosen 
in the complement of a periodic orbit, $a$ and $a'$ will have the same rotation number. 
The problem is that {\em increasing} $a$
on $I$ will {\em decrease} $a^{-1}$ on $a(I)$, and it is not clear if we can find
an adjustment for which the net effect on $w$ is to increase its rotation number.

We abstract the situation in terms of a (one-player) game as follows. 

\begin{definition}
An {\em interval game} consists of a finite collection of orientation-preserving
homeomorphisms $\varphi_1,\cdots,\varphi_m$ (the {\em enemies})
and $\psi$. An interval $I\subset S^1$ 
{\em wins} if there is some positive integer $n$ so that $\psi^i(I)$ is disjoint from
$\varphi_j(I)$ for all $j$ and for $0\le i \le n$, and if $\psi^n(I^+)$ is contained in the
interior of $I$, where $I^+$ denotes the rightmost point of $I$. 
The interval {\em loses} otherwise.

An interval game is {\em constrained} by a finite partition $J$ of $S^1$ if the interval
$I$ must be chosen subject to being entirely contained in one of the intervals of $J$.
\end{definition}

After conjugating $w$, we assume that $w$ ends with $a$, and we 
let $w_1,w_2,\cdots,w_k$ be the (finitely many) suffixes of $w$ that begin with $a^{-1}$
(we take suffixes instead of prefixes because our group acts on the left).

\begin{proposition}
Suppose the interval game has a winning interval, for $\varphi_i = w_i^{-1}$ and $\psi=w$,
constrained relative to the partition consisting of intervals complementary to a finite
orbit for $a$, and suppose $\rotz(w)$ is irrational. Then $R(w,r,s)>\rotz(w)$.
\end{proposition}
\begin{proof}
Let $I$ be a winning interval. We adjust $a$ to $a'$ on $I$, and consider the dynamics of
a point in $I$ under powers of $w'$. We can build a foliation encoding the dynamics on
the mapping torus of $w$ as follows. Let $C=\cup C_i$ be a cylinder with a vertical product foliation,
decomposed into subcylinders each of which represents the dynamics of one letter of $w$.
As we read $w$ from right to left, the subcylinders from bottom to top represent the dynamics
of each successive letter. Then the top of $C$ is glued to the bottom by $w$.

Adjusting $a$ to $a'$ on $I$ can be realized by adjusting the foliation in each subcylinder
associated to an $a$ or $a^{-1}$ in $w$; see Figure~\ref{adjust_foliation}. The figure
shows the altered dynamics on an $a$-subcylinder and an $a^{-1}$-subcylinder. 

\begin{figure}[htpb]
\labellist
\small\hair 2pt
\pinlabel $a$ at -20 40
\pinlabel $a^{-1}$ at -20 160
\pinlabel $w_i$ at 320 100
\pinlabel $I$ at 105 5
\pinlabel $w_i^{-1}(I)$ at 197 5
\pinlabel $\underbrace{\quad \quad \quad}$ at 105 15
\pinlabel $\underbrace{\quad \quad \quad}$ at 197 15
\endlabellist
\centering
\includegraphics[scale=1]{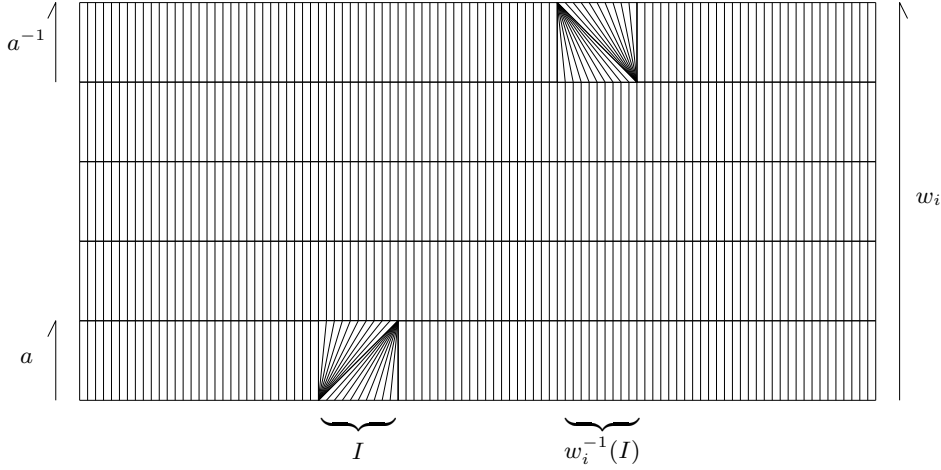}
\caption{Adjusting the dynamics of $a$ on $I$.}\label{adjust_foliation}
\end{figure}

Ignoring all $a$-subcylinders except the first (associated to the terminal $a$ of $w$) 
for the moment, and under the
hypothesis that $I$ is winning, we see that the future itinerary of any
point in $I$ under powers of $w'$ 
is the periodic orbit $w(I^+), w^2(I^+),\cdots, w^n(I^+), w(I^+)$;
i.e.\/ $w'$ has rational rotation number, which is therefore {\em strictly} 
greater than that of $w$.
Taking into account the adjusted dynamics at other $a$-subcylinders can only further
increase the rotation number of $w'$. This completes the proof.
\end{proof}

It is therefore important to understand precisely which interval games have 
a winning interval. First, we recall for the benefit of the reader, some elementary
facts about the orbit of a point under successive powers of an irrational rotation.
Fix some $\theta \in (0,1)$, and a rigid rotation $R_\theta$ through $\theta$.
The case that $\theta\in\Q$ is completely straightforward, so suppose $\theta$ is
irrational. We express $\theta$ as a continued fraction
$$\theta = \frac {1} {a_1 +} \frac {1} {a_2 +} \frac {1} {a_3 +} \cdots$$
Define $\theta_i$ recursively by $\theta_0=\theta$, $\theta_1 = 1-a_1\theta_0$ and
$\theta_{k+1} = \theta_{k-1} - a_{k+1}\theta_k$ for positive $k$.
Fix $r\in S^1$, and let $r_i:=R_\theta^i(r)$ denote the forward orbit. 

The following lemma is straightforward (see e.g.\/ \cite{de_Melo_van_Strien} pp. 26--30):
\begin{lemma}\label{closest_approach_lemma}
Let $r_{j_i}$ for $i=1,2,\cdots$ be the sequence of successively
closest approaches to $r$; i.e.\/ satisfying $|r-r_{j_i}|<|r-r_k|$ for $k<j_i$.
Then the $r_{j_i}$ alternately approach $r$ from the left (if $i$ is odd)
and the right (if $i$ is even), and $|r-r_{j_i}|=\theta_i$.

Moreover, if $r_a,r_b$ are adjacent elements of $\lbrace r_i\; | \; 0\le i\le n \rbrace$ 
with $a<b$ then $|r_a - r_b| = |r - r_{b-a}|$; it follows that there are infinitely
many odd $i$ for which $|r-r_{j_i}| < |r_a-r_b|$ for all adjacent $r_a,r_b$
with $a,b\ne 0$.
\end{lemma}

\begin{definition}
We say $\theta$ is {\em well approximated from the left} if there are a sequence of
odd $i$ for which $|\theta_i|/|\theta_{i-1}| \to 0$.
\end{definition}

The set of $\theta \in S^1$ that are well approximated from the left has full measure; this
is elementary, and follows e.g. by the kinds of estimates proved in \cite{Herman}.

The following theorem says that {\em generic} interval games (in a suitable sense)
have a winning interval.

\begin{theorem}\label{generic_smooth_enemies}
Consider the interval game associated to a collection
$\varphi_1,\cdots,\varphi_m$ of $C^1$ diffeomorphisms, and $\psi$ 
a rigid rotation through angle $\theta$, where $\theta$ is irrational and
well approximated from the left.

Suppose there is a point $p\in S^1$ at which the derivatives of the $\varphi_i$ are all
different from $1$. Then there is a winning interval $I$ contained in any subinterval $J$
sufficiently close to $p$.
\end{theorem}
\begin{proof}
Since the $\varphi_i$ are all $C^1$, and
since their derivatives at $p$ are all different from $1$, we can find an interval
$K$ contained in $J$ so that the $\varphi_i'$ are almost constant
and bounded away from $1$ on $K$. Intuitively, as we adjust the position of $q$ 
near $p$, the $\varphi_i(q)$ move almost linearly at speeds bounded away from $1$. We
can therefore adjust $q$ to a location near $p$ for which none of the $\varphi_i(q)$ are
too close to some $R_\theta^i(q)$.

Let $n$ be very large, and such that $R_\theta^n(p)$ approximates $p$ from the left very well.
We fix some very small $\epsilon$, and require that
$$|p-R_\theta^n(p)| \le \epsilon |R_\theta^i(p)-R_\theta^j(p)|$$
for all $0<i,j<n$. Furthermore, for all $0<i<n$ there should be some $0<j<n$ not equal to $i$
with $|R_\theta^i(p)-R_\theta^j(p)| \ll |K|$.

Because the $\varphi_i'$ are bounded above and below and away from $1$, for each $i$
the set of $q \in K$ for which
$|\varphi_i(q)-R_\theta^j(q)|\le \max_{k\in K}|\varphi_i'(k)|\cdot|q-R_\theta^n(q)|$ for some $j$ has measure of order
$\epsilon|K|$. Choosing $\epsilon \ll 1/m$ we can find some $q$ for which
$|\varphi_i(q)-R_\theta^j(q)|> \max_{k\in K}|\varphi_i'(k)|\cdot|q-R_\theta^n(q)|$ for all $i$ and all $0<j<n$. Then
$[R_\theta^n(q),q]$ is a winning interval.
\end{proof}

A complete analysis of the interval game seems possible but difficult; however, 
we are able to completely understand the special case of a single enemy.

\begin{theorem}\label{single_enemy_nonrotation}
Consider the interval game with a single enemy $\varphi$, and suppose the rotation number
of $\psi$ is irrational. Let $\mu$ be an invariant probability measure for $\psi$. If
$\varphi$ does not preserve $\mu$, there is a winning interval.
\end{theorem}
\begin{proof}
First for simplicity we suppose that $\psi$ is conjugate to a rigid rotation; 
equivalently, that $\mu$ has full support. The graph $\Gamma$ of $\varphi$ 
is a monotone $(1,1)$ curve in the torus $S^1\times S^1$, and by hypothesis, 
it does not have slope $1$ everywhere. 

It follows that we can find a slope $1$ curve $L$ that locally supports $\Gamma$ from above,
and a point $(r,\varphi(r))$ which is locally the rightmost point of $\Gamma \cap L$;
i.e.\/ $\Gamma$ is strictly below $L$ in a neighborhood to the right of this point.
Dynamically, $\varphi$ is {\em strongly contracting}
from the right at $r$; i.e.\/ there is some $\epsilon$ so that for $|[r,s]|\le \epsilon$, 
there is an inequality $|\varphi([r,s])| < |[r,s]|$. In fact, by the strictness of
the inequality, there is an $\epsilon$ so that for all $\delta \le \epsilon$ there
is some $s(\delta)>r$ with $|\varphi([r,s(\delta)])| = |[r,s(\delta)]|-\delta$.
Note that the smallest such $s(\delta)$ with this property has the additional property
that $|\varphi([s,s(\delta)])| < |[s,s(\delta)]|$ for all $r< s<s(\delta)$; i.e.\/
$\varphi$ is strongly contracting from the left at $s(\delta)$.
Geometrically, if we let $L_\delta$ denote the line with slope $1$ obtained by translating
$L$ vertically down $\delta$, then $(s(\delta),\varphi(s(\delta)))$ is the first time
$\Gamma$ crosses $L_\delta$ to the right of $(r,\varphi(r))$; see
Figure~\ref{strongly_contracting}.

\begin{figure}[htpb]
\labellist
\small\hair 2pt
\pinlabel $r$ at 127 -10
\pinlabel $s$ at 172 -10
\pinlabel $s(\delta)$ at 257 -10
\pinlabel $L$ at 20 30
\pinlabel $L_\delta$ at 110 30
\pinlabel $\Gamma$ at 210 190
\pinlabel $\delta$ at 120 100
\pinlabel $|[s,s(\delta)]|$ at 214 150
\pinlabel $|\varphi([s,s(\delta)])|$ at 295 180
\endlabellist
\centering
\includegraphics[scale=0.75]{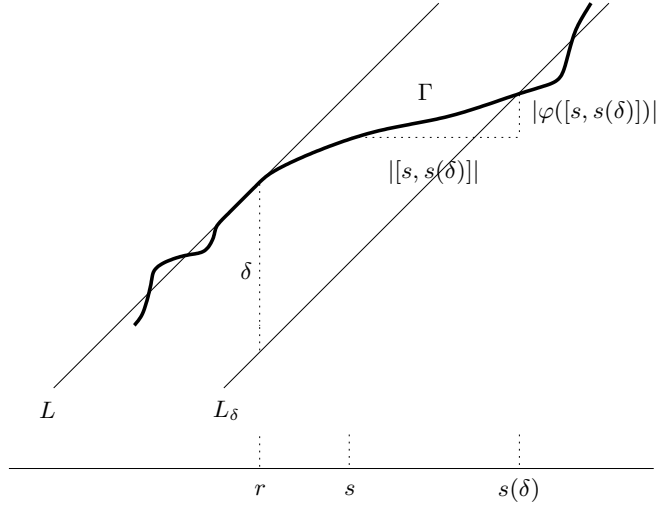}
\caption{$\varphi$ is strongly contracting to the right of $r$.}\label{strongly_contracting}
\end{figure}

We adopt the notation
$r_i:=\psi^i(r)$. Since $\rot(\psi)$ is irrational, by Lemma~\ref{closest_approach_lemma}
we can find an arbitrarily big $m$ so that $r_m$ is a closest approach to $r$ from
the left; in particular, $|r-r_m| < |r_a-r_b|$ for all $0<a,b<m$.
Let $u=|[r_m,r]|$.

Let $r_a,r_b,r_c,r_d$ be successive orbits such that $r_c \le \varphi(r) < r_d$.
Let $t_1=|[r_a,r_b]|$ and $t_2=|[r_b,r_c]|$, and $v=|[r_c,\varphi(r)]|$. Note
that $t_1 > u$ by the definition of $u$. We can
assume (by taking $m$ sufficiently big) that $t_2+v \ll \epsilon$, and 
therefore there is a point $s(t_2+v)$ to the right of $r$ with properties as above.

Let $t=|[r,s(t_2+v)]|$. Then $\varphi(r+t) = r_b+t$, and $|\varphi([r_m+t,r+t])|\le u < t_1$,
so $r_a+t < \varphi(r_m+t) < r_b+t=\varphi(r+t)$. In particular, an interval of
the form $[r_m+t-\delta,r+t]$ is winning, for sufficiently small $\delta$.

\medskip

It remains to consider the general case where $\mu$ does not have full support. Actually, the argument
in this case is essentially the same as that above. In place of $\Gamma$ we can
consider the curve $\Gamma':=\lbrace (\int_0^r d\mu,\int_0^{\varphi(r)} d\mu)\rbrace$. 
This ``graph'' might have horizontal and vertical segments, 
but otherwise we can use the same argument with curves $L,L_\delta$ applied to
$\Gamma'$ in place of $\Gamma$.
\end{proof}

Interestingly enough, the exceptional case that $\varphi$ and $\psi$ are 
both rigid rotations (after a semiconjugacy) turns out to be nontrivial:

\begin{theorem}\label{single_enemy_rotation}
Consider the interval game in which both the single enemy $\varphi$ and $\psi$ are
rigid rotations through $u,t$ respectively. Then there is a winning interval if and
only if $(t,u)$ is contained in an explicit open subset $U$ of the unit square
described below (see Figure~\ref{approximation}).
\end{theorem}
\begin{proof}
Winning $(t,u)$ are classified by which iterate $n$ of $\psi$ certifies the win, and
the smallest integer $m$ such that $nt<m$. Let $U(m)$ be the subset of $U$ with
a given value of $m$. Then one sees directly that $U(1)$ is the union over all $n\ge 2$ and
$1 \le i \le n-1$ of the interior of the triangles with vertices $(\frac 1 n,\frac i n)$, 
$(\frac 1 {n-1},\frac {i-1} {n-1})$, and $(\frac 1 {n-1},\frac i {n-1})$.

On the other hand, by rescaling the interval $[0,m]$ by a factor of $\frac 1 m$ it
is clear that if $(\frac t m, \frac {u+i} m)\in U$ for all $0\le i < m$ then
$(t,u)\in U$, and any element of $U(m)$ is of this kind. This gives a recursive
description of $U$.
\end{proof}

\begin{figure}[htpb]
\labellist
\small\hair 2pt
\endlabellist
\centering
\includegraphics[scale=0.5]{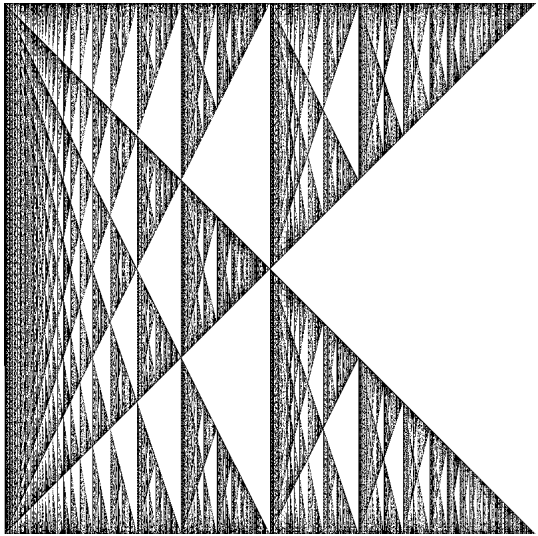}
\caption{The subset $U$ (in white) of the unit square shows the winning interval
games for a pair of rigid rotations.}\label{approximation}
\end{figure}

The complement of $U$ is the attractor of the IFS generated by the set of projective
linear transformations of the form
$$(x,y) \to \left( \frac {x+n} {x+n+1} , \frac y {x+n+1} \right), \quad 
(x,y) \to \left( \frac {x+n} {x+n+1} , \frac {x+y+n} {x+n+1} \right)$$
for all non-negative integers $n$.

Theorems~\ref{single_enemy_nonrotation} and \ref{single_enemy_rotation} together present
an essentially complete picture of the interval game with a single enemy.

Notice for every rational $u$ the set of $t$ with $(t,u)$ in $U$ is an open, dense subset
of $[0,1]$. As a corollary we get strong constraints on $R(w,r,s)$ for $w$ containing at most
one $a^{-1}$:

\begin{corollary}
Suppose $w$ is a word of the form $va^{-1}b^na$ for some $v$ containing no $a^{-1}$,
and suppose $r,s\in \Q$. If $R(w,r,s)$ is irrational then
$(R(w,r,s),ns)$ is not in $U$.

In particular, there is a dense $G_\delta$ subset of $[0,1]$ that $R(w,r,s)$ avoids. 
\end{corollary}

\begin{remark}
For each rational $u$ the set of $t$ with $(t,u)$ not in $U$ is the attractor (i.e.\/ the
limit set) of an explicit finitely generated subsemigroup of $\SL(2,\Z)$ acting projectively on the
interval. For example, if $u=1/2$, the set of ``bad'' $t$ is the limit set of the semigroup
generated by the matrices
$$\begin{pmatrix} 1 & 0 \\ 2 & 1 \\ \end{pmatrix}, \begin{pmatrix} -3 & 2 \\ -8 & 5 \\ \end{pmatrix},
\begin{pmatrix} 1 & 1 \\ 2 & 3 \\ \end{pmatrix}$$
If $T$ is a finitely generated semigroup of contractions of the interval whose images are
disjoint, and some of the maps have a neutral (i.e.\/ parabolic) fixed point, Urba\'nski 
(\cite{Urbanski}; see also \cite{Przytycki_Urbanski}) showed, 
generalizing work of Bowen \cite{Bowen}, that the Hausdorff dimension
of the limit set is the least zero of the pressure function $P$, defined by the
formula 
$$P(s) = \lim_{n \to \infty} \frac 1 n \log \sum_{M \in T_n} \|M'\|^s$$ 
where $T_n$ is the subset of $T$ consisting of words of length $n$, and $\|\cdot\|$ is the
supremum norm. Actually computing this dimension in practice seems hard.
\end{remark}

\section{Acknowledgments}
We would like to thank Kathryn Mann, Shigenori Matsumoto and the anonymous referee
for some useful comments. Danny Calegari was supported by NSF grant DMS 1005246. 

\appendix

\section{Higher rank}

For convenience, we state the analogues of our main theorems
to higher rank. The proofs of these theorems are routine generalizations of the
proofs in the body of the paper, and are omitted.

Throughout this section, let $F$ be a free group of rank $n$ with free generating
set $a_1,a_2,\cdots,a_n$. For real numbers $r_1,r_2,\cdots,r_n$ and $w\in F$, let
$R(w,r_1,r_2,\cdots,r_n)$ denote the maximum of $\rotz(w)$ under representations
$F \to \homeo^+(S^1)^\sim$ for which $\rotz(a_i)=r_i$.

\begin{theorem}[High rank Rationality Theorem]
Suppose $w$ is positive. If the $r_i$ are all rational, so is $R(w,r_1,\cdots,r_n)$.
Moreover, if $w$ has at least one $a_i$,
the denominator of $R(w,r_1,\cdots,r_n)$ is no bigger than that of $r_i$.
\end{theorem}

\begin{theorem}[High rank Stability Theorem]
Suppose $w$ is positive. Then $R$ is locally constant from the right at rational
points; i.e.\/ for every collection of rational numbers $r_i$, there is an
$\epsilon>0$ so that $R(w,\cdots)$ is constant on 
$[r_1,r_1+\epsilon)\times\cdots\times[r_n,r_n+\epsilon)$.

Conversely, if $R(w,r_1,\cdots,r_n)=p/q$ (where $p/q$ is reduced) 
and the biggest power of
consecutive $a_i$s in $w$ is $a_i^{m_i}$ then there is an inequality
$$R(w,r_1,\cdots,r_i+1/m_iq,\cdots,r_n) \ge p/q + 1/q^2$$
\end{theorem}

\begin{theorem}[High rank Stairstep Theorem]
Let $w$ be positive, and suppose we are given rational numbers $p_j/q_j$ for
$j\ne i$ and $c/d$ so that
$$R(w,p_1/q_1,\cdots,t_i,\cdots,p_n/q_n)=c/d$$ for some real $t_i$ 
(so necessarily $d \le q_j$ for each $j$). Then the infimum $t_i$ with this
property is rational, 
and there is an algorithm to compute it. Moreover,
if $p_i/q_i$ is this infimal value, $R(w,p_1/q_1,\cdots,p_n/q_n)=c/d$.
\end{theorem}

\begin{theorem}[High rank Isobar Theorem]
Let $w$ be positive. For any rational $p/q$ the set of $(r_1,\cdots,r_n)\in
[0,1]^n$ 
such that $R(w,r_1,\cdots,r_n)\ge p/q$ is a finite sided rational polyhedron, 
whose boundary consists of finitely many polyhedra on which at least one
$r_i$ is constant.
\end{theorem}

The results in \S~\ref{arbitrary_word_section} also generalize in a straightforward
way to higher rank, but we omit the statements.

\end{document}